\documentclass[a4paper]{article}

\usepackage{graf_GR}
\usepackage{slashed}
\usepackage{esint}

\newcommand{\Riem}{\mathrm{R}}

\newcommand{\AAA}{\Lambda}
\renewcommand{\Fslash}{\slashed{F}}

\newcommand{\Sob}{\mathrm{Sob}}
\renewcommand{\Si}{M}
\newcommand{\prSi}{{\pr\Si}}

\renewcommand{\DDD}{{\mathbb{B}}}
\newcommand{\KG}{\mathrm{K}}

\title{Stability of the Euclidean $3$-ball under $L^2$-curvature pinching}
\author{Olivier Graf\,\footnote{Univ. Grenoble Alpes, CNRS, IF, 38000 Grenoble, France. E-mail: \texttt{olivier.graf@univ-grenoble-alpes.fr}.}}
\begin{document}
\maketitle
\begin{abstract}
  In this article, we consider compact Riemannian $3$-manifolds with boundary. We prove that if the $L^2$-norm of the curvature is small and if the $H^{1/2}$-norm of the difference of the fundamental forms of the boundary is small, then the manifold is diffeomorphic to the Euclidean ball. Moreover, we obtain that the manifold and the ball are metrically close (uniformly and in $H^2$-norm), with a quantitative, optimal bound. The required smallness assumption only depends on the volumes of the manifold and its boundary and on a trace and Sobolev constant of the manifold. The proof only relies on elementary computations based on the Bochner formula for harmonic functions and tensors, and on the $2$-spheres effective uniformisation result of~\cite{Kla.Sze22}.
\end{abstract}

\section{Introduction}
In this article, we consider smooth oriented connected and compact $3$-dimensional Riemannian manifold $(\Si,g)$ with boundary $\pr\Si$. We note $\gd$ the induced Riemannian metric on $\pr\Si$. We note $\th$ the second fundamental form of $\prSi$, defined by 
\begin{align*}
  \th(X,Y) & := g(\nab_XN,Y),  
\end{align*}
for all $X,Y\in T\pr\Si$, where $N$ denotes the outgoing unit normal to $\pr\Si$ and $\nab$ is the Levi-Civita connection associated to $(\Si,g)$. We note $\Nd$ the projection of $\nab$ on $\prSi$, which, for $\prSi$-tangent tensors, coincides with the Levi-Civita connection of $(\prSi,\gd)$. For all smooth $\Si$-tangent tensors $F$, we define the trace norm
\begin{align}\label{eq:defH12}
  \norm{F}_{H^{1/2}(\pr\Si)} & := \inf_{\substack{\text{$\widetilde{F}$ smooth} \\ \text{extension of $F$ in $\Si$}}}\Vert\widetilde{F}\Vert_{H^1(\Si)},
\end{align}
and, for all $n\in\mathbb{N}$,
\begin{align*}
  \norm{F}_{H^{n+1/2}(\prSi)} & := \sum_{m=0}^n\norm{\Nd^mF}_{H^{1/2}(\prSi)}.
\end{align*}
We define $c_{\Sob}$ to be the Sobolev constant given by\footnote{Using Rellich-Kondrachov theorem, $c_{\Sob}$ is finite for any smooth compact manifold. See~\cite{Eva98}.}
\begin{align}\label{eq:defsobconst}
  \begin{aligned}
    c_{\Sob}^{-1} & := \inf_{f \in C^\infty(\Si)} \frac{\norm{\nab f}_{L^{2}(\Si)}}{\Vert f - \bar f \Vert_{L^6(\Si)}},
  \end{aligned}
\end{align}
where $\bar f:=\frac{1}{|\prSi|}\int_\prSi f$. Finally, we denote by $\Riem$ and $\RRRic$, the Riemann and Ricci curvature of $(\Si,g)$.\footnote{For $3$-manifolds, $\Riem$ can be expressed in terms of $\RRRic$. Nevertheless, we make the distinction between the two tensors since there are a few occurrences in the paper where $\Riem$ appears and could not be replaced by $\RRRic$ in dimension greater than 3.}\\

The following theorem is the main result of the paper.
\begin{theorem}\label{thm:main}
  Let $\AAA>0$. There exists $\varep^0>0$, depending only on $\AAA$, such that for all $0<\varep<\varep^0$ the following holds. Any smooth oriented connected and compact 3-dimensional Riemannian manifold $(\Si,g)$ with non-empty boundary, such that 
  \begin{align}\label{est:constants}
     |\Si|,~\frac1{|\prSi|},~c_\Sob,~\norm{N}_{H^{1/2}(\prSi)} \leq \AAA,
  \end{align}
  and
  \begin{subequations}
    \label{est:L2RicTh}
    \begin{align}
      \norm{\Riem}_{L^2(\Si)} & \leq \varep,\\
      \norm{\th-\gd}_{H^{1/2}(\pr\Si)} & \leq \varep,
    \end{align}
  \end{subequations}
  satisfies the following properties:
  \begin{itemize}
  \item $\Si$ is diffeomorphic to the Euclidean $3$-ball $\DDD^3$ \emph{via} a global harmonic coordinate map $\Phi$. In particular, $\pr\Si$ is diffeomorphic to the $2$-sphere $\SSS^2$.
  \item In the coordinates given by $\Phi$, we have
    \begin{align}\label{est:gijdeijthm}
       \norm{g_{ij}-\de_{ij}}_{H^2(\Si)} + \norm{g_{ij}-\de_{ij}}_{L^\infty(\Si)} & \les_\AAA \varep,
    \end{align}
    for all $1\leq i,j \leq 3$.
  \item Moreover, we have
    \begin{align}\label{est:highergijdeijthm}
      \norm{g_{ij}-\de_{ij}}_{H^{n+2}(\Si)} & \les_{\AAA,n} \norm{\Riem}_{H^n(\Si)} + \norm{\th-\gd}_{H^{n+1/2}(\pr\Si)},
    \end{align}
    for all $1\leq i,j \leq 3$ and all $n\geq 0$.
  \end{itemize}
\end{theorem}

\paragraph{Notations.} We use Einstein summation convention and we write $i,j,k,l,\ldots$ the indices used for contractions with $g$ and $a,b,c,\ldots$ the indices used for contractions with $\gd$. We write $A\les_{D}B$ as a shorthand notation for ``there exists a constant $C$ depending only on $D$, such that $A\leq CB$''. 

\paragraph{Remarks on Theorem~\ref{thm:main}.}
\begin{enumerate}
\item In particular, we obtain from setting $\varep=0$ in Theorem~\ref{thm:main} the following \emph{rigidity} result: flat $3$-dimensional manifolds with boundary, with coinciding first and second fundamental forms, are \emph{isometric} to the Euclidean unit $3$-ball.
\item We emphasise that no topological assumptions are made on the manifold $\Si$, apart from the fact that the boundary must be non-empty. Theorem~\ref{thm:main} is of course false in that case, a counter-example being given by the flat torus $\mathbb{T}^3$.
\item The quantitative estimates~\eqref{est:gijdeijthm} are optimal in the sense that, for $\varep$ sufficiently small,~\eqref{est:gijdeijthm} implies $\norm{\Riem}_{L^2(\Si)}\les_\AAA\varep$ and $\norm{\th-\gd}_{H^{1/2}(\prSi)}\les_\AAA\varep$ as assumed in~\eqref{est:L2RicTh}. 
\end{enumerate}

\paragraph{Elliptic PDE heuristic.} Let us first discuss why Theorem~\ref{thm:main} is natural from an elliptic PDE viewpoint. For the classical Poisson equation
\begin{align}\label{eq:Poissonintro}
  \Delta u = f,
\end{align}
multiplication by $u$ and integration by part gives the \emph{energy identity}
\begin{align}\label{eq:energyintro}
  \int_\Si |\nab u|^2 = -\int_\Si fu + \int_\prSi u N(u).
\end{align}
Commuting~\eqref{eq:Poissonintro} by $\nab$, we get
\begin{align}\label{eq:DeltaRicbis}
  \Delta \nab u & = \RRRic\cdot\nab u + \nab f,
\end{align}
which, contracted with $\nab u$ and integrated on $\Si$, using Stokes formula, gives the (integrated) \emph{Bochner formula}
\begin{align}\label{eq:Bochnerintro}
  \begin{aligned}
    & \int_\Si |\nab^2u|^2 + \int_\Si\RRRic \cdot \nab u \cdot \nab u \\
    & \quad \quad + \int_\prSi \th^{ab}\Nd_au\Nd_bu + \int_\prSi \trth (N(u))^2 = \int_\Si |f|^2 + 2\int_\prSi \Nd^au \Nd_a(N(u)).
  \end{aligned}
\end{align}

In the Euclidean case $\RRRic=0$ and $\th=\gd$, combining the energy~\eqref{eq:energyintro} and Bochner~\eqref{eq:Bochnerintro} identities with a Sobolev or Poincar\'e estimate, and classical absorption arguments (see Remark~\ref{rem:absorption}), gives
\begin{align*}
  \norm{u-\bar u}_{H^2(\Si)}^2 & \les \norm{f}^2_{L^2(\Si)} + \int_\prSi \Nd u \cdot \Nd (N(u)) \les \norm{f}^2_{L^2(\Si)} + \norm{\Nd u}_{H^{1/2}(\prSi)} \norm{N(u)}_{H^{1/2}(\prSi)}.  
\end{align*}
Using a trace estimate and an absorption argument, this gives the following \emph{Dirichlet} or \emph{Neumann} elliptic estimate
\begin{align}\label{est:H2demointro}
  \norm{u-\bar u}_{H^2(\Si)} \les \norm{f}_{L^2(\Si)} + \norm{\Nd u}_{H^{1/2}(\prSi)} && \text{or} && \norm{u-\bar u}_{H^2(\Si)} \les \norm{f}_{L^2(\Si)} + \norm{N(u)}_{H^{1/2}(\prSi)}.
\end{align}

On the other hand, \emph{if there exists coordinates $(x^1,x^2,x^3)$ on $\Si$ which are harmonic}, \emph{i.e.} such that $\Delta (x^i) =0$, the metric components in these coordinates satisfy the quasilinear coupled system of Poisson equations
\begin{align}\label{eq:quasilinearPoisson}
  \Delta(g_{ij}) & = -2\RRRic(g)_{ij} + \mathcal{N}_{ij}(g,\pr g),
\end{align}
for all $1 \leq i,j \leq 3$, and where $\mathcal{N}_{ij}(g,\pr g)$ is a non-linearity in $g, \pr g$. Hence, in view of the elliptic estimates~\eqref{est:H2demointro} reviewed above, since $\th \sim N(g)$, it is reasonable to hope for an estimate of the type
\begin{align}\label{est:H2gijdemo}
  \norm{g_{ij}-\de_{ij}}_{H^2(\Si)} & \les \norm{\RRRic}_{L^2(\Si)} + \norm{\th-\gd}_{H^{1/2}(\prSi)}.
\end{align}
In fact, establishing such an estimate is precisely what Theorem~\ref{thm:main} formalises. However, the above heuristic has two caveats: first,~\eqref{eq:quasilinearPoisson} only holds if we know that we already have global harmonic \emph{coordinates} on $\Si$, second, it is not clear how to obtain~\eqref{est:H2gijdemo} from~\eqref{eq:quasilinearPoisson}. In fact, instead of using the elliptic equations~\eqref{eq:quasilinearPoisson} for metric components, we will directly prove elliptic estimates for the harmonic functions $x^i$, using the energy and Bochner formulas~\eqref{eq:energyintro} and~\eqref{eq:Bochnerintro}. See discussions below.

\begin{remark}\label{rem:absorption}
  In this paper, we repeatedly use variations of the following two types of ``absorption arguments'':
  \begin{itemize}
  \item If $A \leq B + \varep C A$, then for $\varep$ sufficiently small depending on $C$ -- here $\varep \leq 1/(2C)$ --, the last term of the RHS can be absorbed on the LHS and one has $A\les B$.
  \item If $A^2 \leq C AB$, then, using that $AB \leq \frac{1}{2C}A^2 + \frac{C}{2}B^2$, absorption by the LHS as above gives $A^2\les_C B^2$. 
  \end{itemize}
\end{remark}

\paragraph{From elliptic estimates to geometric conclusions: an example.}
Let us illustrate on a simple example how geometric conclusions can follow from (elementary) elliptic estimates, such as the ones reviewed in the previous paragraph. The theorem below is a special case of Corollary~\ref{cor:prSiconnected} which is proved in the bulk of this paper.
\begin{theorem*}
  Assume that $\Si=\Omega$ is a connected open subset of $\mathbb{R}^3$ with compact closure and smooth boundary. Assume that $\pr\Omega$ is mean convex, \emph{i.e.} that $\trth>0$. Then $\pr\Omega$ has only one connected component.
\end{theorem*}
\begin{proof}
  Assume that $\pr\Omega$ has at least two connected components $X_1$ and $X_2$. Define $u$ the harmonic function in $\Omega$ such that $u=1$ on $X_1$, $u=-1$ on $X_2$ and $u=0$ on $\pr\Om\setminus X_1\cup X_2$. Using that $\Delta u =0$ and that $u$ is constant on all connected components of $\pr\Om$, the Bochner integral formula~\eqref{eq:Bochnerintro} rewrites
  \begin{align}\label{eq:demoexampleBochner}
   \int_\Om |\nab^2u|^2 + \int_{X_1\cup X_2}\trth (N(u))^2 = 0.
  \end{align}
  Using that $\trth>0$, we deduce that $\nab^2u=0$ in $\Si$ and $N(u)=0$ on $X_1$ and $X_2$. The energy identity~\eqref{eq:energyintro} then yields
  \begin{align*}
    \int_{\Omega}|\nab u|^2 & = \int_{\pr\Om}uN(u) = 0.
  \end{align*}
  Hence $u$ is constant on $\Omega$ and on $\pr\Omega$ which is a contradiction.
\end{proof}

\begin{remark}
  The above proof also highlights the crucial role played by the convexity of the boundary. Namely, in all the Bochner-type formulas of the paper, a term analogous to the second term in~\eqref{eq:demoexampleBochner} will appear, and this term will have a good sign provided that $\trth>0$ in some sense. Note that this holds (in a weak $H^{1/2}$-sense) under the assumptions of Theorem~\ref{thm:main}, since~\eqref{est:L2RicTh} implies that $\trth \simeq 2$.
\end{remark}

\paragraph{Outline of the proof of Theorem~\ref{thm:main}.} The first step is to obtain a series of functional estimates relying only on the bound for the four constants~\eqref{est:constants} and on the curvature and second fundamental form smallness assumption~\eqref{est:L2RicTh}. This is achieved in Section~\ref{sec:func}. To that end we use a ``radius vectorfield'' $X$, defined as the harmonic extension in $\Si$ of the normal to the boundary $N$. The first section (Section~\ref{sec:harmonicradius}) is in fact dedicated to proving estimates for $X$. The key tool there (as in the rest of the paper) is the Bochner formula for harmonic functions and tensors. As a first, notable, byproduct of these preliminaries, we directly obtain that $\prSi$ has only one connected component (Corollary~\ref{cor:prSiconnected}).\\

Using Gauss equation and a trace estimate for the curvature on $\prSi$, we further obtain in Section~\ref{sec:confest} that $\prSi$ is nearly round (in a weak $H^{-1/2}(\prSi)$-sense). Hence, adapting the effective uniformisation result of~\cite{Kla.Sze22}, there exists a conformal isomorphism $\Phi$ between $\prSi$ and the Euclidean sphere $\SSS^2$ which is nearly an isometry. Extending $\Phi$ as an harmonic function on $\Si$ (Section~\ref{sec:prelim}), we have a natural candidate for the global diffeomorphism of Theorem~\ref{thm:main}.\\

To show that this indeed defines a local and global diffeomorphism, the crux is an estimate of the Hessian of $\Phi$ in $L^2(\Si)$ which is, again, based on Bochner formula (Section~\ref{sec:enerBoch}). Once this estimate is closed, controlling higher order derivatives is routine (but technical) for elliptic equations (Section~\ref{sec:nab3xiSi}).\\

The fact that $\Phi$ is a local diffeomorphism follows easily from these estimates (Section~\ref{sec:globdiffeo}). From an application of the maximum principle and simple topological arguments, we also deduce that $\Phi$ is a global diffeomorphism onto the Euclidean ball $\DDD^3$. In Section~\ref{sec:finalproof} we conclude the proof of Theorem~\ref{thm:main}. 

\paragraph{Related works.} We first highlight two ways by which a result close to Theorem~\ref{thm:main} could be obtained:
\begin{itemize}
\item Using a geometric flow (see for example~\cite[Corollary 1.3]{Hui86} where a geometric conclusion is obtained by evolving the mean curvature flow starting from the boundary of the manifold).
\item Using a Cheeger-Gromov compactness argument. See~\cite{Czi19,Czi.Gra22} for such an argument performed under assumptions similar to Theorem~\ref{thm:main} (and see~\cite{Czi19} for further references).
\end{itemize}
In both cases, we believe that the direct proof of the present paper, which relies only on harmonic functions (and tensors), computations, and elementary estimates, is simpler and more satisfactory. In fact, using harmonic functions and (integral) elliptic estimates has proved to be a successful and elegant way of obtaining geometric results. See for example the rigidity and stability results of~\cite{Lic58,Oba62} and~\cite{Pet99b} for the Euclidean $n$-sphere (see also~\cite{Aub03}), which rely on (refinements of) the Bochner formula. See also the recent rigidity result of~\cite{Bra.Kaz.Khu.Ste22} and stability result of~\cite{Don.Son24} in the context of the positive mass theorem.\\

We also point out the fact that the proofs in~\cite{Czi19,Czi.Gra22} rely on contradiction arguments which do not give quantitative estimates, \emph{i.e.} $\varep$ in~\eqref{est:gijdeijthm} should be replaced by an implicit function $\delta(\varep)$ of $\varep$. It could be possible to upgrade from $\delta(\varep)$ to $\varep$ as is done in~\cite{Kla.Sze22} to obtain an effective uniformisation result for the $2$-sphere, although, again, we believe that the present, direct proof is simpler.\\

Finally, we mention that some of the ideas of the present paper are already contained in the appendix of~\cite{Gra20a} (see Appendix A and Theorem 4.4 in that paper). This includes in particular the definition of the harmonic coordinates of Section~\ref{sec:prelim} and associated computations (Proposition~\ref{prop:refinedBochner}), which we rewrote (and simplified) for completeness. On top of many simplifications, we highlight the following main novelties of the present article.
\begin{itemize}
\item In~\cite{Gra20a}, the constant $\varep$ had to be small with respect to many functional constants (in fact, each required functional estimate in~\cite{Gra20a} was assumed to hold with an arbitrary constant). This could be circumvented by using a (strong!) ``weak regularity'' assumption on $\Si$ and $\prSi$ as the ones of~\cite{Sha14}. In the present paper, we drastically reduce the number of constants, requiring only a bound on the volumes of $\Si$ and $\prSi$ and on a trace and Sobolev constant of $\Si$, see~\eqref{est:constants}. This is done thanks to a new functional framework (see for example the definition~\eqref{eq:defH12} of $H^{1/2}(\prSi)$) which is adapted to the situation and allows us to prove each functional estimate needed in the paper, using only the assumption~\eqref{est:constants}.
\item In the present paper, we remove the assumption made in~\cite{Gra20a} that $\Si$ has to be diffeomorphic to $\DDD^3$ and we obtain it \emph{as a conclusion}.
\end{itemize}

\paragraph{Applications to General Relativity.} For $4$-dimensional Lorentzian manifolds $(\MM,\g)$ satisfying the \emph{Einstein equations} $\RRRic(\g)=0$, the second Bianchi identities yield that the spacetime divergence of the Riemann curvature tensor $\Riem(\g)$ vanishes. Hence, the Riemann curvature $\Riem(\g)$ satisfies wave-type energy estimates, in which the $L^2$-norm of $\Riem(\g)$ on spacelike hypersurfaces -- which are $3$-dimensional Riemannian manifolds -- is the natural energy quantity. See~\cite{Chr.Kla93} and the celebrated bounded $L^2$-curvature theorem~\cite{Kla.Rod.Sze15} for further discussions.\\

In this context, the estimates~\eqref{est:gijdeijthm} and~\eqref{est:highergijdeijthm} of Theorem~\ref{thm:main} can be applied to deduce bounds at the metric level from $L^2$-bounds on $\Riem(\g)$. For that, we need that the result of Theorem~\ref{thm:main} is quantitative -- \emph{i.e.} that the RHS of~\eqref{est:gijdeijthm} is linear in $\varep$ -- since in applications, $\varep$ can be a precise (decaying) function of time. We refer to~\cite{Gra20a} for such an application, which was in fact the main motivation for Theorem~\ref{thm:main} (and Theorem 4.4 in~\cite{Gra20a}) and for further discussions. 

\paragraph{Acknowledgements.} This project started after my collaboration with Stefan Czimek whom I thank for interesting discussions. It was partially included in my PhD thesis and I thank my former advisor Jérémie Szeftel for advice and interesting discussions. I also thank Grigorios Fournodavlos, Arick Shao and Jacques Smulevici for stimulating discussions and comments. Finally, I thank Baptiste Devyver for insights on Sobolev estimates on manifolds and other discussions around this project.

\section{The harmonic radius vectorfield $X$}\label{sec:harmonicradius}
\begin{definition}\label{def:X}
  We define the \emph{harmonic radius vectorfield} $X$ to be the smooth vectorfield on $\Si$ such that
  \begin{align}\label{eq:DirichletX}
    \Delta X & = 0 ~ \text{on $\Si$} & X & = N ~ \text{on $\pr\Si$.}
  \end{align}
\end{definition}
The goal of this section is to obtain uniform and $H^2$ estimates for $X$ (see Lemma~\ref{lem:extensionofN} and Proposition~\ref{prop:extensionofNextra}). \emph{En passant}, we prove a general Sobolev estimate (Lemma~\ref{lem:Sobolevenpassant}) which is needed in the proof of Proposition~\ref{prop:extensionofNextra}. The estimates obtained in this section are crucial to prove general functional estimates in Section~\ref{sec:func}. The estimates of Proposition~\ref{prop:extensionofNextra} are also re-used later in the proof of Lemma~\ref{lem:estL2BBB}.

\begin{lemma}\label{lem:extensionofN}
  We have
  \begin{align}
    |X| & \leq 1, \label{est:maxppX}
  \end{align}
  uniformly on $\Si$, and
  \begin{align}
    \norm{\nab X}_{L^2(\Si)} & \leq 2\norm{N}_{H^{1/2}(\pr\Si)}. \label{est:H1X}
  \end{align}
\end{lemma}
\begin{proof}
  From Definition~\ref{def:X}, we have $\Delta\left(|X|^2\right) = |\nab X|^2 \geq 0$ in $\Si$ and $|X|^2=1$ on $\prSi$, and~\eqref{est:maxppX} follows from the maximum principle. Let $\widetilde{N}$ be a smooth extension of $N$ on $\Si$. One has $\Delta (X-\widetilde{N}) = - \Delta \widetilde{N}$ which, contracted with $(X-\widetilde{N})$ and integrated on $\Si$, gives $\Vert\nab(X-\widetilde{N})\Vert_{L^2(\Si)} \leq \Vert\nab \widetilde{N}\Vert_{L^2(\Si)}$, and~\eqref{est:H1X} follows.
\end{proof}
\begin{lemma}\label{lem:Sobolevenpassant}
  For all $\Si$-tangent tensors $F$, we have the Sobolev estimate
  \begin{align}\label{est:SobL6L2H12}
    \begin{aligned}
      \norm{F}_{L^6(\Si)} & \les_{\AAA} \norm{\nab F}_{L^2(\Si)} + \norm{F}_{H^{1/2}(\prSi)}.
    \end{aligned}
  \end{align}
\end{lemma}
\begin{proof}
  Let $f$ be a scalar function on $\Si$ and let $\widetilde{f}$ be a smooth extension of $f|_{\prSi}$ on $\Si$. Let $\widetilde{N}$ be a smooth extension of $N$ on $\Si$. We have
  \begin{align*}
    \begin{aligned}
      \le|\int_{\pr\Si} f\ri| & = \le|\int_{\pr\Si} \widetilde{f} g(\widetilde{N},N)\ri| = \le|\int_\Si \Div\big(\widetilde{f}\widetilde{N}\big)\ri| \leq \int_\Si |\nab \widetilde{f}| |\widetilde{N}| + \int_\Si|\widetilde{f}||\nab\widetilde{N}| \\
      & \leq 2 \Vert \widetilde{f} \Vert_{H^1(\Si)} \Vert \widetilde{N} \Vert_{H^1(\Si)} \\
      & \leq 2 \norm{f}_{H^{1/2}(\prSi)} \norm{N}_{H^{1/2}(\prSi)},
    \end{aligned}
  \end{align*}
  where, to obtain the last inequality, we applied the definition~\eqref{eq:defH12} and took the infimum over all extensions $\widetilde{f}$ and $\widetilde{N}$. Plugging the above in the Sobolev inequality~\eqref{eq:defsobconst}, we have
  \begin{align*}
    \begin{aligned}
      \norm{f}_{L^6(\Si)} & \leq c_{\Sob}\norm{\nab f}_{L^2(\Si)} + \frac{|\Si|^{1/6}}{|\prSi|}\le|\int_\prSi f\ri| \\
                          & \leq c_{\Sob}\norm{\nab f}_{L^2(\Si)} + 2\frac{|\Si|^{1/6}}{|\prSi|}\norm{N}_{H^{1/2}(\prSi)}\norm{f}_{H^{1/2}(\prSi)}.
    \end{aligned}
  \end{align*}
  For $\Si$-tangent tensors $F$, taking $f=|F|$ and using the next remark finishes the proof of~\eqref{est:SobL6L2H12}.
\end{proof}
\begin{remark}
  For all smooth tensors $F$ on $\pr\Si$ we have
  \begin{align}\label{est:|F|H12}
    \norm{|F|}_{H^{1/2}(\pr\Si)} \leq \norm{F}_{H^{1/2}(\pr\Si)},
  \end{align}
  as direct consequence of the definition~\eqref{eq:defH12}.
\end{remark}

\begin{proposition}\label{prop:extensionofNextra}
  Provided that the hypothesis~\eqref{est:L2RicTh} of Theorem~\ref{thm:main} hold with $\varep$ sufficiently small depending on $\AAA$, we have
  \begin{align}\label{est:H2X}
    \norm{\nab^2X}_{L^2(\Si)} + \norm{\nab X-g}_{L^6(\Si)} & \les_{\AAA} \varep.
  \end{align}
  As a consequence, we have
  \begin{align}\label{est:prSi-3Si}
    \bigg||\prSi|-3|\Si|\bigg| & \leq \sqrt{3|\Si|}\norm{\nab X-g}_{L^2(\Si)} \les_\AAA \varep.
  \end{align}
\end{proposition}
\begin{proof}
  We have the commutation formula
  \begin{align}\label{eq:commXbase}
    \Delta \nab X & = \nab \Delta X + g^{ij}\nab_i[\nab_j,\nab]X + g^{ij}[\nab_i,\nab]\nab_jX.
  \end{align}
  Integrating on $\Si$ the above identity contracted with $\nab X$, applying three times Stokes formula, we get
  \begin{align}\label{eq:intnab2uidX}
    \begin{aligned}
      \int_{\Si} |\nab^2X|^2 +\BB & = \int_{\Si}|\Delta X|^2 +\int_\Si g^{ij}[\nab_j,\nab]X\cdot\nab_i\nab X - \int_{\Si} g^{ij}[\nab_i,\nab]\nab_jX \cdot \nab X,
    \end{aligned}
  \end{align}
  where $\BB$ is the following boundary term
  \begin{align}\label{eq:bdytermsintnab2uidX}
    \begin{aligned}
       \BB := & \int_{\pr\Si} \le(\Delta X \cdot\nab_NX - \nab X \cdot \nab_N\nab X + [\nab_N,\nab]X \cdot\nab X\ri) \\
        & = \int_{\pr\Si} \le(\Delta X \cdot\nab_NX - \nab X \cdot \nab\nab_N X \ri) \\
        & = \int_{\pr\Si}\le(\le(\nab_N\nab_NX +\nab_a\nab^aX\ri)\cdot\nab_NX - \nab_NX\cdot\nab_N\nab_NX - \nab_aX\cdot\nab^a\nab_N X\ri) \\
        & = \int_{\pr\Si}\le(\nab_a\nab^aX\cdot\nab_NX - \nab_aX\cdot\nab^a\nab_NX\ri).
    \end{aligned}
  \end{align}
  Since $X=N$ (hence $\nab_aX_b=\th_{ab}$ and $\nab_aX_N=0$) on $\prSi$, we have
  \begin{align*}
    \begin{aligned}
      \int_{\pr\Si}\nab_a\nab^aX\cdot\nab_NX & = \int_{\pr\Si}\nab_a\nab^aX_N\nab_NX_N + \int_{\pr\Si}\nab_a\nab^aX_c\nab_NX^c \\
                                                     & = \int_{\pr\Si}\le(\trth\nab_NX_N - \th^{ac}\th_{bc}\ri)\nab_NX_N \\
                                                     & \quad + \int_{\pr\Si}\le(\Nd^a\th_{ac} +\trth\nab_NX_c\ri)\nab_NX^c \\
                                                     & = \int_{\pr\Si}\le(\trth|\nab_NX|^2 - |\th|^2\nab_NX_N +\Divd\th\cdot\nab_NX\ri) \\
                                                     & = \int_{\pr\Si}\le(\trth|\nab_NX|^2 - |\th|^2\nab_NX_N -\th^{ab}\Nd_a(\nab_NX)_b\ri) \\
                                                     & = \int_{\pr\Si}\le(\trth|\nab_NX|^2 - |\th|^2\nab_NX_N -\th^{ab}\nab_a(\nab_NX)_b +\th^{ab}\th_{ab}\nab_NX_N\ri) \\
                                                     & = \int_{\pr\Si}\le(\trth|\nab_NX|^2 -\th^{ab}\nab_a\nab_NX_b -\th^{ab}\th_{ac}\th^{c}_b\ri).
    \end{aligned}
  \end{align*}
  Plugging the above in~\eqref{eq:bdytermsintnab2uidX} and developing, we have
  \begin{align}\label{eq:bdytermsnab2Xfinal}
    \begin{aligned}
      \BB & = \int_{\pr\Si}\nab_a\nab^aX\cdot\nab_NX - \int_{\pr\Si}\nab_aX\cdot\nab^a\nab_NX \\
          & = \int_{\pr\Si}\le(\trth|\nab_NX|^2 -2\th^{ab}\nab_a\nab_NX_b -\th^{ab}\th_{ac}\th^{c}_b\ri) \\
          & = \int_{\pr\Si}\le(\trth|\nab_NX|^2-2\trth\nab_NX_N-4+4\trth + 2|\th-\gd|^2\ri) - 2 \int_\prSi(\th-\gd)^{ab}\nab_a\nab_NX_b \\
          & \quad -\int_\prSi\le(2 + 3(\trth-2) + 3|\th-\gd|^2 +(\th-\gd)^{ab}(\th-\gd)_{ac}(\th-\gd)^{c}_b\ri) \\
          & = \int_{\pr\Si}\trth|\nab_NX-X|^2 - 2 \int_\prSi(\th-\gd)^{ab}\nab_a\nab_N X_b \\
          & \quad -\int_\prSi\le(|\th-\gd|^2 +(\th-\gd)^{ab}(\th-\gd)_{ac}(\th-\gd)^{c}_b\ri),
    \end{aligned}
  \end{align}
  where, to obtain the third identity, we used that
  \begin{align*}
    -2\int_{\pr\Si}\gd^{ab}\nab_a\nab_NX_b & = -2\int_{\pr\Si}\gd^{ab}\le(\nab_a(\nab_NX)_b-\th_{ac}\th^{c}_b\ri) = -2\int_{\pr\Si}\gd^{ab}\le(\Nd_aY_b+\th_{ab}\nab_NX_N-\th_{ac}\th^{c}_b\ri) \\
                                           & = \int_{\pr\Si}\le(-2\trth\nab_NX_N -4 + 4\trth +2|\th-\gd|^2\ri),
  \end{align*}
  where we denoted by $Y$ the projection of $\nab_NX$ on $T\prSi$.
  
  Noticing that
  \begin{align*}
    \begin{aligned}
      \int_{\Si} g^{ij}[\nab_i,\nab_k]\nab_jX_l\nab^k X^l & = \int_{\Si} g^{ij}[\nab_i,\nab_k](\nab_jX_l-g_{jl})\nab^k X^l.
    \end{aligned}
  \end{align*}
  and plugging~\eqref{eq:bdytermsnab2Xfinal} in~\eqref{eq:intnab2uidX}, we obtain
  \begin{align*}
    \begin{aligned}
      \int_\Si|\nab^2 X|^2 + 2 \int_\prSi |\nab_NX-X|^2 & = \int_\Si g^{ij}[\nab_j,\nab]X\cdot\nab_i\nab X -\int_{\Si} g^{ij}[\nab_i,\nab_k](\nab_jX_l-g_{jl})\nab^k X^l\\
                                                        & \quad + \int_{\pr\Si}\bigg(-(\trth-2)|\nab_NX-X|^2 + 2(\th-\gd)^{ab}\nab_a\nab_N X_b \\
                                                        & \quad \quad \quad \quad \quad \quad +|\th-\gd|^2 + (\th-\gd)\cdot(\th-\gd)\cdot(\th-\gd)\bigg).
    \end{aligned}
  \end{align*}
  Hence, using~\eqref{est:maxppX}, we infer that
  \begin{align}\label{est:nab2XnabXgfinal}
    \begin{aligned}
      & \norm{\nab^2X}^2_{L^2(\Si)} + \norm{\nab X-g}_{L^2(\prSi)}^2 \\
      \les & \; \norm{\Riem}_{L^2(\Si)}\le(\norm{\nab X-g}_{L^2(\Si)} + \norm{\nab X-g}^2_{L^4(\Si)} + \norm{\nab^2X}_{L^2(\Si)}\ri) + \EE_{\prSi}
    \end{aligned}
  \end{align}
  with
  \begin{align*}
    \EE_\prSi & := \int_{\pr\Si}\bigg(-(\trth-2)|\nab_NX-X|^2 + 2(\th-\gd)^{ab}\nab_a\nab_N X_{b}\bigg) + \norm{\th-\gd}_{L^2(\prSi)}^2 + \norm{\th-\gd}_{L^3(\prSi)}^3.
  \end{align*}
  Let $\Th$ be a smooth extension of $\th-\gd$ in $\Si$. From~\eqref{est:SobL6L2H12}, we have
  \begin{align}\label{est:proofSobL4Th}
    \begin{aligned}
      \norm{\Th}_{L^6(\Si)} & \les_\AAA \norm{\Th}_{H^1(\Si)}, & \norm{\Th}_{L^4(\Si)} & \leq |\Si|^{1/12}\norm{\Th}_{L^6(\Si)} \les_\AAA \norm{\Th}_{H^1(\Si)}.
    \end{aligned}
  \end{align}
  Using Stokes formula, Cauchy-Schwarz,~\eqref{eq:DirichletX}, and~\eqref{est:proofSobL4Th}, and that $\Th$ is $\prSi$-tangent, we have
  \begin{align}\label{est:ThnabnabXprSi}
    \begin{aligned}
      \int_{\prSi} (\th-\gd)^{ab}\nab_a\nab_N X_{b} & = \int_{\prSi} \Th^{ij}\nab_i\nab_N X_j \\
                                                    & = \int_\Si \le(\nab^{k}\Th^{ij}\nab_i\nab_k X_j + \Th^{ij}\nab^k\nab_i\nab_k X_j\ri) \\
                                                    & \leq \norm{\Th}_{H^1(\Si)}\norm{\nab^2X}_{L^2(\Si)} + \int_\Si \Th^{ij}[\nab^k,\nab_i](\nab_kX_j-g_{kj}) \\
                                                    & \les \norm{\Th}_{H^1(\Si)}\norm{\nab^2X}_{L^2(\Si)} + \norm{\Riem}_{L^2(\Si)}\norm{\Th}_{L^4(\Si)}\norm{\nab X-g}_{L^4(\Si)} \\
                                                    & \les_\AAA \norm{\Th}_{H^1(\Si)}\bigg(\norm{\nab^2X}_{L^2(\Si)} + \norm{\Riem}_{L^2(\Si)}\norm{\nab X-g}_{L^6(\Si)}\bigg).
    \end{aligned}
  \end{align}
  Similarly, using additionally~\eqref{est:H1X}, we have
  \begin{align}\label{est:ThnabXnabXprSi}
    \begin{aligned}
      \int_{\pr\Si}(\trth-2)|\nab_NX-X|^2 & = \int_\prSi\tr\Th\le|\le(\nab X-g\ri)(X)\ri|^2 \\
                                          & = \int_\Si \Div\le(\tr\Th(\nab_XX-X)(\nab X-g)\ri) \\
                                          & = \int_\Si (\nab_XX-X)\cdot \nab\tr\Th \cdot (\nab X-g) \\
                                          & \quad + \int_\Si \tr\Th\Div\le((\nab_XX-X)\cdot(\nab X-g)\ri) \\
                                          & \les \norm{\Th}_{H^1(\Si)}\norm{\nab X-g}_{L^4(\Si)}^2 + \norm{\Th}_{L^4(\Si)}\norm{\nab X-g}_{L^4(\Si)}\norm{\nab^2X}_{L^2(\Si)} \\
                                          & \quad + \norm{\Th}_{L^6(\Si)}\norm{\nab X}_{L^2(\Si)}\norm{\nab X-g}^2_{L^6(\Si)} \\
                                          & \les_\AAA \norm{\Th}_{H^1(\Si)}\bigg(\norm{\nab X-g}_{L^6(\Si)}^2 + \norm{\nab X-g}_{L^6(\Si)}\norm{\nab^2X}_{L^2(\Si)}\\
                                          & \quad\quad\quad\quad\quad\quad\quad\quad + \norm{\nab X-g}_{L^6(\Si)}^2\bigg),                         
    \end{aligned}
  \end{align}
  and
  \begin{align}\label{est:ThL3prSi}
    \begin{aligned}
    \norm{\th-\gd}_{L^3(\prSi)}^3 & = \int_{\prSi}|\Th|^3g(X,N) = \int_{\Si}\Div\le(|\Th|^3X\ri) \\
                                  & = \int_{\Si}3|\Th|^2\nab_X\Th + \int_{\Si}|\Th|^3(\Div X) \\
                                  & \les \norm{\Th}_{L^4(\Si)}^2\norm{\Th}_{H^1(\Si)} + \norm{\Th}^3_{L^6(\Si)}\norm{\nab X}_{L^2(\Si)} \\
                                  & \les_\AAA \norm{\Th}_{H^1(\Si)}^3, 
    \end{aligned}
  \end{align}
  and
  \begin{align}\label{est:ThL2prSi}
    \norm{\th-\gd}_{L^2(\prSi)}^2 & \les_\AAA \norm{\Th}_{H^1(\Si)}^2.
  \end{align}
  Combining~\eqref{est:proofSobL4Th} in~\eqref{est:ThnabnabXprSi}, \eqref{est:ThnabXnabXprSi}, \eqref{est:ThL3prSi}, and taking the infimum for all extensions $\Th$ of $\th-\gd$, using~\eqref{eq:defH12}, we have
  \begin{align}\label{est:EEprSiradius}
    \begin{aligned}
      \EE_\prSi & \les_\AAA \norm{\th-\gd}_{H^{1/2}(\prSi)} \bigg(\norm{\nab^2X}_{L^2(\Si)} + \norm{\nab X-g}^2_{L^6(\Si)} +\norm{\nab X-g}_{L^6(\Si)}\norm{\nab^2X}_{L^2(\Si)} \\
      & \quad\quad\quad\quad\quad\quad\quad + \norm{\Riem}_{L^2(\Si)}\norm{\nab X-g}_{L^6(\Si)} \bigg) + \norm{\th-\gd}^2_{H^{1/2}(\prSi)} + \norm{\th-\gd}^3_{H^{1/2}(\prSi)}.
    \end{aligned}
  \end{align}
  The Sobolev estimate~\eqref{eq:defsobconst} for $f=|\nab X-g|$ yields
  \begin{align}\label{est:SobnabXg}
    \begin{aligned}
    \norm{\nab X-g}_{L^6(\Si)} & \leq c_{\Sob}\norm{\nab^2X}_{L^2(\Si)} + \frac{|\Si|^{1/6}}{|\prSi|^{1/2}} \norm{\nab X-g}_{L^2(\prSi)}.
    \end{aligned}
  \end{align}
  Combining~\eqref{est:nab2XnabXgfinal}, \eqref{est:EEprSiradius} and~\eqref{est:SobnabXg}, using the hypothesis~\eqref{est:L2RicTh} and that one can take $\varep<1$ (so that $\varep^3<\varep^2<\varep$), we obtain
  \begin{align*}
    \begin{aligned}
      & \norm{\nab^2X}_{L^2(\Si)}^2 + \norm{\nab X-g}_{L^6(\Si)}^2 + \norm{\nab X-g}^2_{L^2(\prSi)} \\
      \les_\AAA \; &  \varep\bigg(\norm{\nab X-g}_{L^6(\Si)} + \norm{\nab X-g}^2_{L^6(\Si)} \\
      & \quad + \le(1+\norm{\nab X-g}_{L^6(\Si)}\ri)\norm{\nab^2X}_{L^2(\Si)}\bigg) + \varep^2.
    \end{aligned}
  \end{align*}
  The desired~\eqref{est:H2X} follows by an absorption argument, provided that $\varep$ is small with respect to $\AAA$. The inequality~\eqref{est:prSi-3Si} then follows from Stokes formula:
  \begin{align*}
    \big||\prSi|-3|\Si|\big| & = \le|\int_{\prSi}g(X,N) - \int_\Si 3\ri| = \le|\int_\Si (\Div X-3)\ri| \leq \sqrt{3}\int_\Si|\nab X-g|
  \end{align*}
  together with Cauchy-Schwarz, and this finishes the proof of the proposition.
\end{proof}

\begin{remark}\label{rem:|prSi||Si|-1}
  Given the assumptions~\eqref{est:constants}, we deduce from~\eqref{est:prSi-3Si} that $|\prSi| \les_\AAA 1$ and $1/|\Si| \les_\AAA 1$ provided that $\varep$ is sufficiently small depending on $\AAA$. In the rest of the paper, we will therefore always silently bound $|\Si|, 1/|\Si|, |\prSi|$, and $1/|\prSi|$ by $\AAA$, without explicitly displaying them in the estimates.  
\end{remark}

\section{Functional estimates}\label{sec:func}
The goal of this section is to prove, starting only from the trace norm definition~\eqref{eq:defH12}, the Sobolev constant definition~\eqref{eq:defsobconst} and the curvature and second fundamental form bounds~\eqref{est:L2RicTh}, a series of general (Sobolev and elliptic) estimates. More precisely, we prove
\begin{itemize}
\item an $H^1(\Si)\xhookrightarrow{} L^6(\Si)$ estimate in Lemma~\ref{lem:firstfunctionalestimates} and an $H^2(\Si)\xhookrightarrow{} L^\infty(\Si)$ estimate in Lemma~\ref{lem:Linfty},
\item an $H^{1/2}(\prSi)\xhookrightarrow{}L^4(\prSi)$ estimate in Lemma~\ref{lem:firstfunctionalestimates} and an $H^{3/2}(\prSi)\xhookrightarrow{}L^\infty(\prSi)$ estimate in Lemma~\ref{lem:ellipticboundary},
\item an elliptic estimate on the boundary $\Ld^{-1}: H^{-1/2}(\prSi)\xrightarrow{}H^{3/2}(\prSi)$ in Lemma~\ref{lem:ellipticboundary},
\item elliptic and product estimates in Section~\ref{sec:technicalSi}. 
\end{itemize}
In fact, the estimates on $\prSi$ follow naturally from the ones on $\Si$ since we defined $H^{1/2}$ as a trace norm~\eqref{eq:defH12}. The bridge between these estimates relies on Stokes formula and on the harmonic extension $X$ of the normal to $\prSi$, for which we have obtained a suitable control in Section~\ref{sec:harmonicradius}. The functional estimates proved in this section are far from being exhaustive. For concision, we only proved what is needed in the rest of the article.\\

Last but not least, a remarkable consequence of our estimates is that $\prSi$ cannot have more than one connected component (Corollary~\ref{cor:prSiconnected}).

\subsection{Sobolev estimates}
\begin{lemma}\label{lem:firstfunctionalestimates}
  Provided that the hypothesis of Theorem~\ref{thm:main} hold with $\varep$ sufficiently small, we have
  \begin{align}
    \norm{F}_{L^2(\pr\Si)} & \les_\AAA \norm{F}_{L^4(\prSi)} \les_\AAA \norm{F}_{H^{1/2}(\pr\Si)},\label{est:L2H12pf}
  \end{align}
  and
  \begin{align}
    \norm{F}_{L^6(\Si)} & \les_\AAA \norm{\nab F}_{L^2(\Si)} + \norm{F}_{L^2(\Si)},\label{est:sobeucl1bis}
  \end{align}
  for all $\Si$-tangent tensors $F$.
\end{lemma}
\begin{proof}
  Estimate~\eqref{est:sobeucl1bis} is a direct consequence of~\eqref{est:SobL6L2H12} and~\eqref{eq:defH12}. Let $f$ be a scalar function on $\prSi$. The first inequality in~\eqref{est:L2H12pf} is a consequence of Cauchy-Schwarz and~\eqref{est:prSi-3Si} (see also Remark~\ref{rem:|prSi||Si|-1}). Arguing as in the proof of Proposition~\ref{prop:extensionofNextra}, using~\eqref{est:maxppX}, \eqref{est:H2X} and~\eqref{est:sobeucl1bis} and noting $\widetilde{f}$ a smooth extension of $f$ on $\Si$, we have
  \begin{align*}
    \begin{aligned}
      \norm{f}_{L^4(\pr\Si)}^4 & = \int_\prSi |f|^4 g(X,N) = \int_\Si \Div\le(|\widetilde{f}|^4X\ri) \\
      & \les \Vert \widetilde{f} \Vert^3_{L^6(\Si)}\Vert \nab \widetilde{f} \Vert_{L^2(\Si)} + |\Si|^{1/6} \Vert \widetilde{f} \Vert^4_{L^6(\Si)} \norm{\nab X}_{L^6(\Si)} \\
      & \les_\AAA \Vert \widetilde{f} \Vert_{H^1(\Si)}^4.
    \end{aligned}
  \end{align*}
  Taking the infimum on all $\widetilde{f}$, the second inequality in~\eqref{est:L2H12pf} follows. For tensors we take $f=|F|$ and use~\eqref{est:|F|H12}.
\end{proof}

From the above estimates, we infer the following corollary.
\begin{corollary}\label{cor:decompperpH12}
  Let $F$ be a $\Si$-tangent $(k,l)$-tensor on $\prSi$. Let $F'$ be any projection of $F$ on $T\prSi$ or $T\prSi^\perp$. We have
  \begin{align}\label{est:F'FH12}
    \norm{F'}_{H^{1/2}(\prSi)} \les_{\AAA,k,l} \norm{F}_{H^{1/2}(\prSi)}.
  \end{align}
  In particular, we have
  \begin{align}\label{est:gdH12}
    \norm{\gd}_{H^{1/2}(\prSi)} & \les_\AAA \norm{g}_{H^{1/2}(\prSi)} \les_\AAA 1.
  \end{align}
\end{corollary}
\begin{proof}
  Take $F$ a $1$-form and $\widetilde{F}$ a smooth extension of $F$ on $\Si$. Using~\eqref{est:maxppX}, \eqref{est:sobeucl1bis}, and~\eqref{est:H2X}, we have
  \begin{align*}
    \norm{F_N}_{H^{1/2}(\prSi)} & \leq \Vert \widetilde{F}_X\Vert_{H^1(\Si)} \les \Vert \widetilde{F} \Vert_{H^1(\Si)} + \Vert \widetilde{F} \Vert_{L^4(\Si)} \norm{\nab X}_{L^4(\Si)} \les_\AAA \Vert \widetilde{F} \Vert_{H^1(\Si)}.
  \end{align*}
  Taking the infimum on all smooth extensions $\widetilde{F}$, using~\eqref{eq:defH12}, the desired~\eqref{est:F'FH12} follows for $1$-forms. Arguing similarly for any $(k,l)$-tensors, this finishes the proof of the corollary.
\end{proof}

We have the following $L^\infty$-Sobolev estimates on $\Si$.
\begin{lemma}\label{lem:Linfty}
  Provided that the hypothesis of Theorem~\ref{thm:main} hold with $\varep$ sufficiently small, we have
  \begin{align}
    \norm{F}_{L^\infty(\Si)} & \les_\AAA \norm{F}_{H^2(\Si)}.\label{est:sobeucl2bis}
  \end{align}
\end{lemma}
\begin{proof}
  The proof that~\eqref{est:sobeucl2bis} follows from~\eqref{est:sobeucl1bis} by iteration is classical, see~\cite[p.157]{Gil.Tru01}. We reproduce it here, for convenience of the reader. Let $F$ be a tensor on $\Si$ and define
  \begin{align}\label{eq:renormalisingFproofGilTru}
    \widetilde{F} & :=  (1+|\Si|)^{-1}C^{-1}\le(\norm{\nab F}_{L^6(\Si)} + \norm{F}_{L^6(\Si)}\ri)^{-1} F,
  \end{align}
  where $C$ is the constant in the inequality~\eqref{est:sobeucl1bis} (which depends on $\AAA$). For $k\in\mathbb{N}$, applying~\eqref{est:sobeucl1bis} to $|\widetilde{F}|^{2^k}$, using H\"older estimate and~\eqref{eq:renormalisingFproofGilTru}, we have
  \begin{align*}
    \begin{aligned}
      \Vert \widetilde{F}\Vert^{2^k}_{L^{6\times2^k}(\Si)} & \leq C 2^k \Vert \widetilde{F} \Vert_{L^{3\times(2^{k}-1)}(\Si)}^{(2^k-1)}\Vert\nab \widetilde{F}\Vert_{L^6(\Si)} + C \Vert\widetilde{F}\Vert^{2^k}_{L^{2\times2^k}(\Si)} \\
                                                           & \leq 2^k \Vert\widetilde{F}\Vert_{L^{6\times2^{k-1}}(\Si)}^{(2^k-1)} \underbrace{|\Si|^{2^{-k}/3}}_{\leq 1+|\Si|} C\le(\Vert\nab \widetilde{F}\Vert_{L^6(\Si)} + \Vert\widetilde{F}\Vert_{L^6(\Si)}\ri) \\
                                           & \leq 2^k \Vert\widetilde{F}\Vert_{L^{6\times2^{k-1}}(\Si)}^{(2^k-1)}.
    \end{aligned}
  \end{align*}
  By iteration, using~\eqref{eq:renormalisingFproofGilTru}, we infer
  \begin{align*}
    \begin{aligned}
      \Vert\widetilde{F}\Vert_{L^{6\times2^n}(\Si)} & \leq \exp\le(\log(2)\le(\sum_{k=0}^n k 2^{-k} \ri)\ri) \Vert\widetilde{F}\Vert_{L^6(\Si)}^{\prod_{k=1}^n(1-2^{-k})} \les_\AAA 1.
    \end{aligned}
  \end{align*}
  Hence, using that $\Vert\widetilde{F}\Vert_{L^{6\times2^n}(\Si)} \xrightarrow{n\to+\infty} \Vert\widetilde{F}\Vert_{L^\infty(\Si)}$ and using~\eqref{eq:renormalisingFproofGilTru}, we deduce that
  \begin{align*}
    \norm{F}_{L^\infty(\Si)} & \les_\AAA \norm{\nab F}_{L^6(\Si)} + \norm{F}_{L^6(\Si)},
  \end{align*}
  from which, using~\eqref{est:sobeucl1bis}, the desired~\eqref{est:sobeucl2bis} follows. 
\end{proof}

\subsection{Elliptic estimates on $\prSi$}\label{sec:ellipticprSi}
\begin{definition}\label{def:H-12}
  For all $\Si$-tangent tensor $F$, we define the dual $H^{-1/2}(\prSi)$ norm by
  \begin{align*}
    \norm{F}_{H^{-1/2}(\prSi)} & := \sup_{G\neq0} \le(\frac{1}{\norm{G}_{H^{1/2}(\prSi)}} \int_\prSi F\cdot G\ri).
  \end{align*}
\end{definition}
\begin{remark}
  From~\eqref{est:L2H12pf}, we have
  \begin{align*}
    \int_{\pr\Si} F\cdot G \leq \norm{F}_{L^2(\pr\Si)}\norm{G}_{L^2(\pr\Si)} \les_\AAA \norm{F}_{L^2(\pr\Si)}\norm{G}_{H^{1/2}(\pr\Si)},
  \end{align*}
  hence
  \begin{align}\label{est:H-12L2}
    \norm{F}_{H^{-1/2}(\pr\Si)} \les_\AAA \norm{F}_{L^2(\pr\Si)}.
  \end{align}
\end{remark}

\begin{lemma}[Elliptic estimates on $\pr\Si$]\label{lem:ellipticboundary}
  Provided that the hypothesis of Theorem~\ref{thm:main} are satisfied with $\varep$ sufficiently small, the following holds. For all function $u$ on $\pr\Si$ with $\int_{\pr\Si} u = 0$, we have
  \begin{align}\label{est:H32H-12ell}
    \norm{u}_{H^{3/2}(\pr\Si)} + \norm{u}_{L^\infty(\prSi)} & \les_\AAA \norm{\Ld u}_{H^{-1/2}(\pr\Si)},
  \end{align}
  where here and in the sequel, we denote by $\Ld$ the Laplace-Beltrami operator on $(\prSi,\gd)$.
\end{lemma}
\begin{proof}
  Let $u$ on $\pr\Si$ be such that $\int_{\pr\Si}u =0$ and extend $u$ as an harmonic function on $\Si$. We recall the energy identity~\eqref{eq:energyintro}:
  \begin{align}\label{eq:energproofell}
    \norm{\nab u}^2_{L^2(\Si)} & = \int_{\pr\Si} u N(u),
  \end{align}
  and the Sobolev inequality~\eqref{eq:defsobconst}, which using that $\int_{\pr\Si} u =0$, gives 
  \begin{align}\label{eq:Sobproofell}
    \norm{u}_{L^6(\Si)} & \leq c_{\Sob} \norm{\nab u}_{L^2(\Si)}.
  \end{align}
  Combining~\eqref{eq:energproofell} and~\eqref{eq:Sobproofell}, and using~\eqref{eq:defH12},~\eqref{est:L2H12pf} and an absorption argument, we infer
  \begin{align}\label{est:lowfreqproofell}
    \norm{u}_{L^2(\Si)}^2 + \norm{\nab u}_{L^2(\Si)}^2 & \les_\AAA \norm{N(u)}_{L^2(\pr\Si)}^2. 
  \end{align}
  The Bochner identity~\eqref{eq:Bochnerintro} gives
  \begin{align}\label{eq:Bochnerproofell}
    \begin{aligned}
      \norm{\nab^2u}^2_{L^2(\Si)} + \int_{\prSi} |\Nd u|^2 + 2 \int_{\prSi} (N(u))^2 & = -\int_{\Si} \RRRic\cdot\nab u \cdot\nab u - 2 \int_\prSi N(u)\Ld u \\
                                                                                     & \quad - \int_\prSi (\th-\gd)\cdot\Nd u\cdot\Nd u - \int_{\pr\Si} (N(u))^2\left(\trth-2\right).
    \end{aligned}
  \end{align}
  Hence, combining~\eqref{est:lowfreqproofell}, \eqref{eq:Bochnerproofell}, and using~\eqref{est:prSi-3Si}, we have
  \begin{align}\label{est:almostfinalproofu}
    \begin{aligned}
      & \quad \norm{u}_{L^2(\Si)}^2 + \norm{\nab u}^2_{L^2(\Si)} + \norm{\nab^2u}^2_{L^2(\Si)} + \norm{\Nd u}_{L^2(\prSi)}^2 + \norm{N(u)}_{L^2(\pr\Si)}^2 \\
      \les_\AAA & \;  - \int_{\Si} \RRRic\cdot\nab u \cdot\nab u - 2\int_{\pr\Si}N(u)\Ld u - \int_{\pr\Si}\le(\th-\gd\ri)\cdot\Nd u\cdot\Nd u - \int_{\pr\Si} (N(u))^2\left(\trth-2\right) \\
      \les_\AAA & \; \norm{\RRRic}_{L^2(\Si)}\norm{\nab u}^2_{L^4(\Si)} + \norm{N(u)}_{H^{1/2}(\prSi)}\norm{\Ld u}_{H^{-1/2}(\pr\Si)} \\
      & + \norm{\th-\gd}_{L^4(\prSi)}\le(\norm{\Nd u}_{L^4(\prSi)}^2 + \norm{N(u)}_{L^4(\prSi)}^2\ri).
    \end{aligned}
  \end{align}
  From Corollary~\ref{cor:decompperpH12} and~\eqref{eq:defH12}, we have
  \begin{align}\label{est:nabuprooftrace}
    \norm{\Nd u}_{H^{1/2}(\prSi)} + \norm{N(u)}_{H^{1/2}(\prSi)} & \les_\AAA \norm{\nab u}_{L^2(\Si)} + \norm{\nab^2u}_{L^2(\Si)}.
  \end{align}
  Hence, using~\eqref{est:nabuprooftrace}, the Sobolev estimates~\eqref{est:L2H12pf} and~\eqref{est:sobeucl1bis} and~\eqref{est:L2RicTh}, an absorption argument in~\eqref{est:almostfinalproofu} yields
  \begin{align*}
    \norm{u}_{H^2(\Si)}^2  + \norm{u}_{H^{1/2}(\pr\Si)}^2 + \norm{\Nd u}_{H^{1/2}(\pr\Si)}^2 + \norm{N(u)}^2_{H^{1/2}(\pr\Si)} & \les_\AAA \norm{\Ld u}_{H^{-1/2}(\pr\Si)},  
  \end{align*}
  provided that $\varep$ is sufficiently small. The $L^\infty$-Sobolev estimate~\eqref{est:sobeucl2bis} then yields
  \begin{align*}
    \norm{u}_{L^\infty(\prSi)} \leq \norm{u}_{L^\infty(\Si)} \les_\AAA \norm{u}_{H^2(\Si)} \les_\AAA \norm{\Ld u}_{H^{-1/2}(\prSi)},
  \end{align*}
  which finishes the proof of~\eqref{est:H32H-12ell}.
\end{proof}

\begin{corollary}\label{cor:prSiconnected}  
  $\prSi$ has only one connected component.
\end{corollary}
\begin{proof}
  Assume that $\prSi$ has at least two connected components $Z_1$ and $Z_2$. Define the smooth function $u$ on $\prSi$ by $u = \frac{1}{|Z_1|} 1_{Z_1} - \frac{1}{|Z_2|} 1_{Z_2}$ where $1_Z$ denotes the indicator function of $Z$. Then $\Ld u =0$ and $\int_\prSi u =0$, thus, by~\eqref{est:H32H-12ell}, $u=0$ which is a contradiction.
\end{proof}
\begin{corollary}\label{cor:WPDelta}
  For all smooth scalar function $f$ with $\int_{\prSi} f = 0$, there exists a unique smooth scalar function $v$ such that $\int_\prSi v=0$ and $\Ld v=f$ on $\prSi$. 
\end{corollary}
\begin{proof}
  Since $\prSi$ is smooth, compact, and connected (by Corollary~\ref{cor:prSiconnected}), there exists a constant $C>0$ such that the Poincar\'e estimate
  \begin{align*}
    \norm{v}_{L^2(\prSi)}\leq C\norm{\Nd v}_{L^2(\prSi)}
  \end{align*}
  holds for all $v$ with $\int_\prSi v=0$ (note that the above is trivially false for manifolds with several connected components). Thus, $(v,w)\mapsto \int_\prSi \Nd v \cdot \Nd w$ is a coercive scalar product on the space of $H^1(\prSi)$-functions with vanishing average and the corollary follows classically from Riesz representation theorem and elliptic regularity.  
\end{proof}

\begin{lemma}\label{lem:H12H-12int}
  For all scalar functions $u$ with $\int_\prSi u =0$ on $\pr\Si$, we have
  \begin{align}\label{est:H12L2H-12}
    \norm{u}_{H^{1/2}(\pr\Si)} & \les_\AAA \norm{\Nd u}_{H^{-1/2}(\pr\Si)} \les_\AAA \norm{\Nd u}_{L^2(\prSi)}.
  \end{align}
\end{lemma}
\begin{proof}
  Let $u$ be a scalar function with $\int_\prSi u =0$ on $\pr\Si$ and extend $u$ as an harmonic function in $\Si$. By Stokes formula, we have $\int_{\pr\Si}N(u) = \int_\Si \Delta u = 0$. Hence, by Corollary~\ref{cor:WPDelta}, there exists $v$ on $\prSi$ such that $\int_{\pr\Si}v=0$ and $\Ld v = N(u)$. Using Lemma~\ref{lem:ellipticboundary}, we have
  \begin{align}\label{est:vNuproof}
    \norm{v}_{H^{1/2}(\pr\Si)} + \norm{\Nd v}_{H^{1/2}(\pr\Si)} &\les_\AAA \norm{N(u)}_{H^{-1/2}(\pr\Si)}.
  \end{align}
  Combining the energy identity~\eqref{eq:energyintro} and the Sobolev estimate~\eqref{eq:defsobconst} for $u$ and~\eqref{est:vNuproof}, and integrating by part on $\prSi$, we have
  \begin{align}\label{est:almostfinalNduH12proof}
    \norm{u}_{L^2(\Si)}^2 + \norm{\nab u}^2_{L^2(\Si)} & \les_\AAA \int_{\prSi}u\Ld v \les_\AAA \le|\int_\prSi \Nd u \Nd v\ri| \les_\AAA \norm{\Nd u}_{H^{-1/2}(\prSi)} \norm{N(u)}_{H^{-1/2}(\prSi)}.
  \end{align}
  Now, for all functions $w$ on $\prSi$, we have by Stokes formula and~\eqref{eq:defH12}
  \begin{align*}
    \int_{\prSi}w N(u) & = \int_\Si \Div\big( \widetilde{w} \nab u \big) = \int_\Si \nab \widetilde{w} \cdot\nab u \leq \norm{w}_{H^{1/2}(\prSi)} \norm{\nab u}_{L^2(\Si)},
  \end{align*}
  where we took the infimum on all smooth extensions $\widetilde{w}$ of $w$ in $\Si$. Hence $\norm{N(u)}_{H^{-1/2}(\prSi)} \leq \norm{\nab u}_{L^2(\Si)}$. Plugged in~\eqref{est:almostfinalNduH12proof} together with an absorption argument, this yields
  \begin{align*}
    \norm{u}_{L^2(\Si)} + \norm{\nab u}_{L^2(\Si)} & \les_\AAA \norm{\Nd u}_{H^{-1/2}(\pr\Si)},
  \end{align*}
  which, using~\eqref{eq:defH12}, proves the first inequality in~\eqref{est:H12L2H-12}. The second inequality is a consequence of~\eqref{est:H-12L2}.
\end{proof}
\begin{corollary}
  We have
  \begin{align}\label{est:H12L2H-12bis}
    \norm{u}_{H^{1/2}(\prSi)} \les_\AAA \norm{\Nd u}_{H^{-1/2}(\prSi)} + \norm{u}_{L^2(\prSi)} \les_\AAA \norm{u}_{H^1(\prSi)},
  \end{align}
  for all scalar functions $u$ on $\prSi$.
\end{corollary}

\subsection{Elliptic estimates on $\Si$}\label{sec:technicalSi}
The estimates of this section will be used to prove Proposition~\ref{prop:nabxi3Si}. This section can be skipped at first reading.
\begin{lemma}\label{lem:highelliptic}
  Under the assumptions of Theorem~\ref{thm:main}, provided that $\varep$ is sufficiently small, we have
  \begin{align}\label{est:highelliptic}
    \begin{aligned}
      \norm{\nab^2U}^2_{L^2(\Si)} & \les_\AAA \norm{\Delta U}^2_{L^2(\Si)} + \varep\norm{\nab U}^2_{L^2(\Si)} + \varep^2\norm{U}^2_{L^2(\Si)} \\
                                  & \quad + \int_{\pr\Si}\le(\nab^aU\cdot\nab_a\nab_NU-\nab^a\nab_aU\cdot\nab_NU\ri),
    \end{aligned}
  \end{align}
  for all $\Si$-tangent $(k,l)$-tensors $U$.
\end{lemma}
\begin{proof}
  Arguing along the same lines as in the proof of Proposition~\ref{prop:extensionofNextra}, one has
  \begin{align}\label{eq:intnab2uid}
    \begin{aligned}
      \int_{\Si} |\nab^2U|^2 & = \int_{\Si}|\Delta U|^2 - \int_{\Si} g^{ij}[\nab_i,\nab]\nab_jU \cdot \nab U +\int_\Si g^{ij}[\nab_j,\nab]U\cdot\nab_i\nab U \\
      & \quad - \int_{\pr\Si}[\nab_N,\nab]U \cdot\nab U - \int_{\pr\Si} \Delta U \cdot\nab_NU + \int_{\pr\Si} \nab U \cdot \nab_N\nab U.
    \end{aligned}
  \end{align}
  We have 
  \begin{align}\label{eq:bdytermsintnab2uid}
    \begin{aligned}
       \BB & := \int_{\pr\Si} \le(-\Delta U \cdot\nab_NU + \nab U \cdot \nab_N\nab U - [\nab_N,\nab]U \cdot\nab U\ri) \\
        & = \int_{\pr\Si}\le(\nab_aU\cdot\nab^a\nab_NU-\nab_a\nab^aU\cdot\nab_NU\ri).
    \end{aligned}
  \end{align}
  Hence, we infer from~\eqref{eq:intnab2uid} and~\eqref{est:L2RicTh} that
  \begin{align*}
    \norm{\nab^2U}^2_{L^2(\Si)} & \les \norm{\Delta U}^2_{L^2(\Si)} + \varep\norm{\nab U}^2_{L^4(\Si)} + \varep\norm{U}_{L^\infty(\Si)}\norm{\nab^2U}_{L^2(\Si)} + \BB.
  \end{align*}
  Using the Sobolev estimates~\eqref{est:sobeucl1bis}, \eqref{est:sobeucl2bis} and an absorption argument, we obtain the desired~\eqref{est:highelliptic}.
\end{proof}

We have the following trace estimate.
\begin{lemma}
  For all $\Si$-tangent tensors $F$, we have
  \begin{align}\label{est:traceH-12N}
    \norm{\Nd\nab_NF}_{H^{-1/2}(\prSi)} & \les_\AAA \norm{\nab^2F}_{L^2(\Si)} + \varep \norm{\nab F}_{L^2(\Si)}. 
  \end{align}
\end{lemma}
\begin{proof}
  Let $G$ be a $\Si$-tangent tensor. Let $\slashed{G}$ be the projection of $G$ onto $\prSi$ in its first component, \emph{i.e.} such that $\slashed{G}_{N\cdots}=0$ and denote by $\widetilde{\slashed{G}}$ a smooth extension of $\slashed{G}$ on $\Si$. Using~\eqref{est:sobeucl1bis}, we have
  \begin{align*}
    \begin{aligned}
      \int_\prSi G^a\cdot\Nd_a\nab_NF & = \int_\prSi \slashed{G}^j\cdot\nab_j\nab_NF \\
      & = \int_\Si \nab^i \widetilde{\slashed{G}}^j \cdot \nab_j\nab_iF + \int_\Si \widetilde{\slashed{G}}^j \cdot [\nab^i,\nab_j] \nab_i F + \int_\Si \widetilde{\slashed{G}}^j \cdot \nab_j \Delta F \\
      & = \int_\Si \nab^i \widetilde{\slashed{G}}^j \cdot \nab_j\nab_iF + \int_\Si \widetilde{\slashed{G}}^j \cdot [\nab^i,\nab_j] \nab_i F - \int_\Si \nab_j\widetilde{\slashed{G}}^j \cdot \Delta F + \int_\prSi \underbrace{\slashed{G}_N}_{=0}\cdot\Delta F \\
      & \les_\AAA \big\Vert \widetilde{\slashed{G}} \big\Vert_{H^{1}(\Si)}\le(\norm{\nab^2F}_{L^2(\Si)} + \varep \norm{\nab F}_{L^2(\Si)}\ri).
    \end{aligned}
  \end{align*}
  Hence, using~\eqref{eq:defH12} and Corollary~\ref{cor:decompperpH12}, we obtain the desired~\eqref{est:traceH-12N}.
\end{proof}

We have the following product estimates on $\prSi$.
\begin{lemma}
  For all $\prSi$-tangent $(k,l)$-tensors $F,G$ on $\prSi$, we have
  \begin{align}\label{est:productLinftyH12}
    \norm{F\cdot G}_{H^{1/2}(\prSi)} & \les_{\AAA,k,l} \norm{F}_{H^{3/2}(\prSi)}\norm{G}_{H^{1/2}(\prSi)}.
  \end{align}
\end{lemma}
\begin{proof}
  For simplicity we consider $F$ a $k$-cotangent tensor on $\prSi$. Estimate~\eqref{est:productLinftyH12} will follow from the definition~\eqref{eq:defH12} and Sobolev estimates~\eqref{est:sobeucl1bis}, \eqref{est:sobeucl2bis} provided that one can construct an extension $F$ of $F$ on $\Si$ such that $\norm{F}_{H^2(\Si)}\les_{\AAA,k,l}\norm{F}_{H^{3/2}(\prSi)}$. To that end, we define $F$ as the harmonic extension of $F$ on $\Si$. Let $\widetilde{F}$ be a smooth extension of $F$ on $\Si$. One has $\Delta (F-\widetilde{F}) = - \Delta \widetilde{F}$, hence contracting with $(F-\widetilde{F})$ and integrating on $\Si$, we infer that $\Vert\nab(F-\widetilde{F})\Vert_{L^2(\Si)} \leq \Vert\nab \widetilde{F}\Vert_{L^2(\Si)}$. Taking the infimum over all smooth extensions $\widetilde{F}$, we have by the definition~\eqref{eq:defH12}
  \begin{align}\label{est:FGH1}
    \norm{F}_{H^1(\Si)}  & \les \norm{F}_{H^{1/2}(\prSi)}.
  \end{align}
  Applying the formula~\eqref{est:highelliptic}, we have
  \begin{align}\label{est:basicFnab2}
    \norm{\nab^2F}^2_{L^2(\Si)} & \les_\AAA \varep \norm{F}^2_{H^1(\Si)} + \int_\prSi \nab^aF\cdot\nab_a\nab_NF - \nab^a\nab_aF\cdot\nab_NF.
  \end{align}
  Using that $F$ is $\prSi$-tangent, we have
  \begin{align*}
    \begin{aligned}
      -\int_\prSi\nab^a\nab_aF\cdot \nab_NF & = -\int_\prSi \Nd^a\Nd_aF\cdot\nab_NF + \trth|\nab_NF|^2 = \int_\prSi \Nd_aF\cdot\Nd^aY - \trth|\nab_NF|^2,
    \end{aligned}
  \end{align*}
  where $Y$ is the projection on $(T^\ast\prSi)^{\otimes k}$ of $\nab_NF$, and we note that
  \begin{align*}
    \begin{aligned}
      \Nd^aY_{a_1\cdots a_k} & = \nab^a\nab_NF_{a_1\cdots a_k} + \th^{ac}\nab_cF_{a_1\cdots a_k} - \sum_{j=1}^k \th^{a}_{a_{j}}\nab_NF_{a_1\cdots N \cdots a_k}. 
    \end{aligned}
  \end{align*}
  Plugged in~\eqref{est:basicFnab2}, using~\eqref{est:H-12L2}, \eqref{est:L2H12pf}, \eqref{est:L2RicTh} and~\eqref{est:gdH12}, we get
  \begin{align*}
    \norm{\nab^2F}^2_{L^2(\Si)} + 2\int_\prSi |\nab_NF|^2 & \les_\AAA \varep\norm{F}^2_{H^1(\Si)} + \norm{\Nd F}_{H^{1/2}(\prSi)}\le(\norm{\Nd\nab_NF}_{H^{-1/2}(\prSi)} + \norm{\Nd F}_{H^{1/2}(\prSi)}\ri) \\
                                                          & \quad + k \norm{\Nd F}_{H^{1/2}(\prSi)}\norm{\nab_NF}_{H^{1/2}(\prSi)} + \varep \norm{\nab_NF}_{H^{1/2}(\prSi)}^2.
  \end{align*}
  Using~\eqref{eq:defH12},~\eqref{est:traceH-12N},~\eqref{est:FGH1} and an absorption argument, provided that $\varep$ is sufficiently small, we get
  \begin{align*}
    \norm{\nab^2F}_{L^2(\Si)} & \les_{\AAA,k} \norm{F}_{H^{3/2}(\prSi)},
  \end{align*}
  as desired, and this finishes the proof of the lemma.
\end{proof}

\section{Uniformisation of $\pr\Si$}\label{sec:confest}
In Section~\ref{sec:func} we obtained that $\prSi$ has only one connected component. In this section, we show that $\prSi$ is conformally isometric to the Euclidean sphere $\SSS^2$ with conformal factor close to $1$.
\begin{definition}\label{def:confiso}
  We say that a smooth map $\Phi:\pr\Si \to \SSS^2$ is a \emph{conformal isomorphism (between $(\pr\Si,\gd)$ and $(\SSS^2,\gd_{\SSS})$)} if:
  \begin{itemize}
  \item The map $\Phi$ defines a diffeomorphism between $\pr\Si$ and $\SSS^2$,
  \item There exists a \emph{conformal factor}, \emph{i.e.} a map $\phi:\pr\Si\to \RRR_+^\ast$ such that
    \begin{align*}
      \phi^{-2}\gd & = \Phi^\ast\gd_{\SSS},
    \end{align*}
    where ${}^\ast$ denotes the pull-back operation.
  \end{itemize}
\end{definition}

We have the following uniformisation theorem.
\begin{proposition}[Uniformisation theorem on $\pr\Si$]\label{lem:unifthm2}
  Under the assumptions of Theorem~\ref{thm:main} and provided that $\varep$ is sufficiently small, there exists a conformal isomorphism $\Phi$ with conformal factor $\phi$ such that
  \begin{align}\label{est:assconf}
    \norm{\log\phi}_{H^{3/2}(\prSi)} + \norm{\log\phi}_{L^\infty(\prSi)} + \norm{\phi-1}_{H^{3/2}(\pr\Si)} + \norm{\phi-1}_{L^\infty(\pr\Si)} & \les_\AAA \varep.
  \end{align}
  Moreover, for all $k\geq 0$, we have the following higher regularity estimates
  \begin{align}\label{est:higherregconf}
    \norm{\log\phi}_{H^{k+3/2}(\prSi)} + \norm{\phi-1}_{H^{k+3/2}(\pr\Si)} & \les_{\AAA,k}\norm{\Riem}_{H^k(\Si)} + \norm{\th-\gd}_{H^{k+1/2}(\pr\Si)}.
  \end{align}
\end{proposition}
\begin{proof}
  Define the Einstein tensor
  \begin{align*}
    \mathrm{G} & := \RRRic-\half (\tr_g\RRRic)\, g.
  \end{align*}
  Recall that $\mathrm{G}$ satisfies the Bianchi identity
  \begin{align}\label{eq:Bianchi}
    \Div \mathrm{G} & = 0,
  \end{align}
  and that Gauss equation reads
  \begin{align}\label{eq:Gaussconf}
    \KG & = -\mathrm{G}_{NN} +\quar \trth^2 -\half |\thh|^2 = 1 -\mathrm{G}_{NN} + (\trth-2) + \quar (\trth-2)^2 -\half |\thh|^2,
  \end{align}
  where $\KG$ denotes the Gauss curvature of $\pr\Si$.
  
  For all functions $f$ on $\pr\Si$, noting $\widetilde{f}$ an extension of $f$ in $\Si$, using Bianchi equation~\eqref{eq:Bianchi}, and the estimates~\eqref{est:maxppX}, \eqref{est:H2X} on $X$, we have
  \begin{align*}
    \int_{\pr\Si} f\mathrm{G}_{NN} & = \int_{\Si} \nab^i\left(\widetilde{f} \mathrm{G}_{Xi}\right) = \int_{\Si} \left(\mathrm{G} \cdot X \cdot \nab \widetilde{f} + \widetilde{f} \mathrm{G} \cdot \nab X\right) \les_\AAA \norm{\mathrm{G}}_{L^2(\Si)}\Vert\widetilde{f}\Vert_{H^1(\Si)}.
  \end{align*}
  Hence, taking the infimum on all extensions $\widetilde{f}$ of $f$ on $\Si$ and taking the supremum for all functions $f$ on $\pr\Si$, we infer that 
  \begin{align}\label{est:H12GNN}
    \begin{aligned}
    \norm{\mathrm{G}_{NN}}_{H^{-1/2}(\pr\Si)} & \les_\AAA \norm{\mathrm{G}}_{L^2(\Si)}.
    \end{aligned}
  \end{align}
  From Gauss equation~\eqref{eq:Gaussconf}, the estimate~\eqref{est:H12GNN}, using~\eqref{est:H-12L2}, the Sobolev estimates~\eqref{est:L2H12pf}, and~\eqref{est:L2RicTh}, we have
  \begin{align}\label{est:H12Kdesired}
    \norm{\KG-1}_{H^{-1/2}(\pr\Si)} & \les_\AAA \norm{\mathrm{G}_{NN}}_{H^{-1/2}(\prSi)} + \norm{\trth-2}_{L^2(\prSi)} + \norm{\th-\gd}^2_{L^4(\prSi)} \les_\AAA \varep.
  \end{align}

  Using Bianchi equation~\eqref{eq:Bianchi} and the estimate~\eqref{est:H2X} for $X$, we have
  \begin{align*}
    \left|\overline{\mathrm{G}_{NN}}\right| & = \frac1{|\prSi|} \left|\int_\prSi \mathrm{G}_{XN}\right| = \frac1{|\prSi|} \left|\int_{\Si} \Div\left(\mathrm{G} \cdot X\right)\right| \leq \frac1{|\prSi|} \norm{\mathrm{G}}_{L^2(\Si)} \norm{\nab X}_{L^2(\Si)} \les_\AAA \norm{\mathrm{G}}_{L^2(\Si)}.
  \end{align*}
  Using Gauss formula~\eqref{eq:Gaussconf}, estimate~\eqref{est:L2RicTh} and~\eqref{est:L2H12pf}, this yields
  \begin{align}\label{est:overK}
    \overline{\KG-1} & \les_\AAA \varep.
  \end{align}

  By Corollary~\ref{cor:WPDelta}, one can now define $u$ to be the solution of
  \begin{align*}
    \Ld u & = \KG-1 - \overline{(\KG-1)}, & \overline{u} & = 0,
  \end{align*}
  on $\prSi$. From~\eqref{est:H12Kdesired},~\eqref{est:overK} and the elliptic estimate~\eqref{est:H32H-12ell}, we infer in particular
  \begin{align}\label{est:confu}
    \norm{u}_{L^\infty(\pr\Si)} & \les_\AAA \varep. 
  \end{align}
  Defining $\widetilde{\gd} := e^{2u}\gd$, we have
  \begin{align*}
    \widetilde{\KG} & = e^{-2u}(\KG - \Ld u) = e^{-2u}(1+\overline{\KG-1}),
  \end{align*}
  where $\widetilde{\KG}$ denotes the Gauss curvature of $(\prSi,\widetilde{\gd})$. Thus, using~\eqref{est:overK} and~\eqref{est:confu}, we have
  \begin{align}\label{est:Kreg+Linf}
    \begin{aligned}
    \big\Vert\widetilde{\KG}-1\big\Vert_{L^\infty(\pr\Si)} & \les_\AAA \varep.
    \end{aligned}
  \end{align}
  Moreover, we have that~\eqref{est:overK} implies that if $\varep$ is sufficiently small depending on $\AAA$, $\overline{\KG} > 0$. Hence, by Gauss-Bonnet formula, $\prSi$ is of genus $0$, and by the classical uniformisation theorem, $\prSi$ is diffeomorphic to $\SSS^2$. Using this fact and the bound~\eqref{est:Kreg+Linf}, we can apply the result of~\cite[Theorem 3.1]{Kla.Sze22}, and there exists a conformal isomorphism $\widetilde{\Phi}$ with conformal factor $\widetilde{\phi}$ between $(\pr\Si,\widetilde{\gd})$ and $(\SSS^2,\gd_\SSS)$ such that
  \begin{align}\label{est:widetildephiLinf}
    \begin{aligned}
      \big\Vert\widetilde{\phi}-1\big\Vert_{L^\infty(\pr\Si)} & \les_\AAA\varep.
    \end{aligned}
  \end{align}
  Defining $\phi := e^{-u}\widetilde{\phi}$, $\Phi:=\widetilde{\Phi}$ is a conformal isomorphism between $(\pr\Si,\gd)$ and $(\SSS^2,\gd_{\SSS})$ with conformal factor $\phi$, and from~\eqref{est:confu} and~\eqref{est:widetildephiLinf} we have
  \begin{align}\label{est:phiLinfconf}
    \norm{\phi-1}_{L^\infty(\prSi)} & \les_\AAA\varep.
  \end{align}
  Moreover, the conformal factor $\phi$ satisfies the elliptic equation
  \begin{align}\label{eq:confKphi}
    \Ld\log\phi & = -\KG + \phi^{-2} = -(\KG-1) +(\phi^{-2}-1).
  \end{align}
  Using the estimate~\eqref{est:H12Kdesired} for $\KG-1$, and~\eqref{est:H-12L2},~\eqref{est:phiLinfconf} to estimate $\phi$, this yields $\norm{\Ld\log\phi}_{H^{-1/2}(\prSi)}\les_\AAA\varep$. Using the elliptic estimate~\eqref{est:H32H-12ell} and the fact that $\phi-1$ is uniformly small, the desired~\eqref{est:assconf} follows. The higher regularity estimates~\eqref{est:higherregconf} are obtained by taking a $\Nd$ derivative in Gauss equation~\eqref{eq:Gaussconf} and following the above steps. The details are left to the reader and this finishes the proof of the proposition. 
\end{proof}

\section{The harmonic functions $x^i$}\label{sec:prelim}
In this section, we define our candidates for the harmonic coordinates $x^i$ of Theorem~\ref{thm:main} and collect some preliminary identities.
\begin{definition}\label{def:coordsxi}
  Let $\Phi$ be a conformal isomorphism between $\prSi$ and $\SSS^2$ (see Definition~\ref{def:confiso}). We define $x^1,x^2,x^3$ to be the unique smooth functions on $\Si$ such that
  \begin{align}\label{eq:defcoordsxi}
    \begin{aligned}
      \Delta x^i = 0 ~ \text{on $\Si$}, && x^i = x^i_{\SSS}\circ\Phi~\text{on $\prSi$},
    \end{aligned}
  \end{align}
  where, on the right-hand side, $x^i_\SSS$ denote the standard Cartesian functions on $\SSS^2$.
\end{definition}
\begin{definition}\label{def:Grammatrix}
  Let $\Phi$ be a conformal isomorphism with conformal factor $\phi$ and let $x^i$ be the associated functions of Definition~\ref{def:coordsxi}. We define $\BBB$ to be the following $2$-tensor on $\Si$
  \begin{align*}
    \BBB & := \sum_{i=1}^3\nab x^i \otimes \nab x^i - g.
  \end{align*}
\end{definition}

We have the following identities for $x^i$ on the boundary $\prSi$.
\begin{lemma}\label{lem:conformalidentities}
  Let $\Phi$ be a conformal isomorphism with conformal factor $\phi$ and let $x^i$ be the associated functions of Definition~\ref{def:coordsxi}. On $\prSi$, we have
  \begin{align}\label{eq:sumxi}
    \sum_{i=1}^3 (x^i)^2 & = 1,  
  \end{align}
  and
  \begin{align}\label{eq:sumnabxi}
    \sum_{i=1}^3\Nd x^i\otimes\Nd x^i = \phi^{-2}\gd,
  \end{align}
  and, for all $1\leq i \leq 3$,
  \begin{align}\label{eq:confHessianxi}
    \begin{aligned}
      \Nd^2x^i & = -x^i\phi^{-2}\gd - 2\Nd(\log\phi)\otimesh\Nd x^i,
    \end{aligned}
  \end{align}
  with $\otimesh$ the symmetrised, traceless, tensor product.
\end{lemma}
\begin{proof}
  Let us denote by $x^i_{\mathbb{R}}$ the Cartesian coordinates of $\mathbb{R}^3$ (so that with the notations of Definition~\ref{def:coordsxi}, $x^i_{\mathbb{R}}|_{\SSS^2} = x^i_\SSS$). The first identity~\eqref{eq:sumxi} follows from the fact that $\sum_{i=1}^3(x^i_{\mathbb{R}})^2=1$ on $\SSS^2$. Since $(\nab x^i_{\mathbb{R}})_{i=1,2,3}$ is an orthonormal frame of $\mathbb{R}^3$ and since $\pr_rx^i_{\mathbb{R}} = x^i_{\mathbb{R}}$ on $\SSS^2$, we have
  \begin{align*}
    N^\flat\otimes N^\flat + \gd_\SSS = g_{\mathbb{R}} = \sum_{i=1}^3 \nab x^i_{\mathbb{R}}\otimes \nab x^i_{\mathbb{R}} & = \le(\sum_{i=1}^3(x^i_{\mathbb{R}})^2 \ri) N^\flat\otimes N^\flat + \sum_{i=1}^3\Nd x^i_{\mathbb{R}}\otimes\Nd x^i_{\mathbb{R}} \\
    & = N^\flat\otimes N^\flat + \sum_{i=1}^3\Nd x^i\otimes\Nd x^i,
  \end{align*}
  and~\eqref{eq:sumnabxi} follows using that $\gd_\SSS=\phi^{-2}\gd$. We have the conformal formula for the Hessian
  \begin{align*}
    \begin{aligned}
      \Nd^2x^i & = \Nd^{2}x^i_{\mathbb{R}} - 2\Nd(\log\phi)\otimesh\Nd x^i.
    \end{aligned}
  \end{align*}
  Using that $\nab^{2}x^i_{\mathbb{R}}=0$, we infer that $\Nd^{2}x^i_{\mathbb{R}} = -x^i_{\mathbb{R}}\gd_\SSS$, which plugged in the above formula gives~\eqref{eq:confHessianxi}.
\end{proof}

We have the following refinement of the Bochner identity~\eqref{eq:Bochnerintro}.
\begin{proposition}\label{prop:refinedBochner}
  Let $\Phi$ be a conformal isomorphism with conformal factor $\phi$ and let $x^i$ be the associated functions of Definition~\ref{def:coordsxi}. We have
  \begin{align}\label{eq:refinedBochner}
    \begin{aligned}
      \sum_{i=1}^3\norm{\nab^2 x^i}^2_{L^2(\Si)} + 2 \sum_{i=1}^3 \int_{\pr\Si} \le(N(x^i)-x^i\ri)^2 & = \EE,
    \end{aligned}
  \end{align}
  with
  \begin{align*}
    \EE & :=  \half\int_{\pr\Si}\le(\trth-2\ri)^2 - \int_{\pr\Si}|\thh|^2 - \int_{\pr\Si}\le(\phi^{-2}-1\ri)\le(\trth-2\ri) \\
        & \quad + \sum_{i=1}^3 \int_{\pr\Si}4\le(\phi^{-2}-1\ri)x^i\le(N(x^i)-x^i\ri) - \sum_{i=1}^3\int_{\pr\Si} \le(\trth-2\ri) \le(N(x^i)-x^i\ri) \le(N(x^i)+x^i\ri)\\
       & \quad  + 2\sum_{i=1}^3 \int_{\pr\Si}  \mathrm{G}_{NN}x^i \le(N(x^i)-x^i\ri) - 2\sum_{i=1}^3\int_{\Si}x^i \mathrm{G}\cdot \nab^2x^i -3\int_{\Si}\le(\RRRic-\frac{1}{3}(\tr_g\RRRic)\, g\ri)\cdot\BBB,
  \end{align*}
  where $\BBB$ is the tensor of Definition~\ref{def:Grammatrix}.
\end{proposition}
\begin{proof}
  Summing the Bochner formulas~\eqref{eq:Bochnerintro} for each harmonic functions $x^i$, we have
  \begin{align}\label{eq:reproofBochner0}
    \begin{aligned}
      \sum_{i=1}^3\int_\Si |\nab^2x^i|^2  & = - \sum_{i=1}^3\int_\Si\RRRic \cdot \nab x^i \cdot \nab x^i + 2\sum_{i=1}^3\int_\prSi \Nd^ax^i \Nd_a(N(x^i)) \\
      & \quad -\sum_{i=1}^3\int_\prSi \th^{ab}\Nd_ax^i\Nd_bx^i -\sum_{i=1}^3\int_\prSi \trth (N(x^i))^2.
    \end{aligned}
  \end{align}
  Introducing the Einstein tensor $\mathrm{G}$ and using that $\RRRic-\frac13(\tr_g\RRRic) g$ is traceless, we have
  \begin{align}\label{eq:reproofBochner1}
    \begin{aligned}
      -\sum_{i=1}^3\int_\Si \RRRic\cdot\nab x^i\cdot\nab x^i & = -3\int_\Si\le(\RRRic-\frac13(\tr_g\RRRic) g\ri)\cdot\BBB + 2 \sum_{i=1}^3 \int_\Si \mathrm{G}\cdot\nab x^i\cdot\nab x^i.
    \end{aligned}
  \end{align}
  Using Stokes formula, Bianchi equation~\eqref{eq:Bianchi} and~\eqref{eq:sumxi}, the last integral in~\eqref{eq:reproofBochner1} rewrites
  \begin{align}\label{eq:reproofBochner2}
    \begin{aligned}
       2 \sum_{i=1}^3 \int_\Si \mathrm{G}\cdot\nab x^i\cdot\nab x^i & = -2\sum_{i=1}^3\int_\Si x^i \mathrm{G}\cdot\nab^2x^i +2\sum_{i=1}^3\int_\prSi \mathrm{G}_{NN}x^iN(x^i) \\
      & = -2\sum_{i=1}^3\int_\Si x^i \mathrm{G}\cdot\nab^2x^i +2\int_\prSi \mathrm{G}_{NN} + 2\sum_{i=1}^3\int_\prSi \mathrm{G}_{NN}x^i\le(N(x^i)-x^i\ri).
    \end{aligned}
  \end{align}
  We note that using Gauss equation~\eqref{eq:Gaussconf} combined with~\eqref{eq:confKphi}, the integral of $\mathrm{G}_{NN}$ in~\eqref{eq:reproofBochner2} rewrites
  \begin{align}\label{eq:reproofBochner2Gauss}
    \begin{aligned}
      2\int_\prSi \mathrm{G}_{NN} & = - 2 \int_\prSi \phi^{-2} + \half \int_\prSi \trth^2 - \int_\prSi |\thh|^2.
    \end{aligned}
  \end{align}
  Using the conformal formula~\eqref{eq:sumnabxi}, we have
  \begin{align}\label{eq:reproofBochner3}
    \begin{aligned}
      -\sum_{i=1}^3\int_\prSi \th^{ab}\Nd_ax^i\Nd_bx^i & = -\int_\prSi \phi^{-2}\trth.
    \end{aligned}
  \end{align}
  Integrating by part and using the conformal formula~\eqref{eq:confHessianxi} (which gives $\Ld x^i = - 2\phi^{-2} x^i$ on $\prSi$), we have
  \begin{align}\label{eq:reproofBochner4}
    \begin{aligned}
      2 \sum_{i=1}^3\int_\prSi \Nd^ax^i\Nd_a(N(x^i)) & = 4\sum_{i=1}^3\int_\prSi \phi^{-2}x^iN(x^i). 
    \end{aligned}
  \end{align}
  Using~\eqref{eq:sumxi}, a direct computation gives
  \begin{align}\label{eq:reproofBochner5}
    \begin{aligned}
      &  - 2 \phi^{-2} + \half \trth^2 - |\thh|^2 - \phi^{-2}\trth + 4 \sum_{i=1}^3\phi^{-2}x^iN(x^i) -\sum_{i=1}^3\trth(N(x^i))^2 \\
      = \; & \half\le(\trth-2\ri)^2 - |\thh|^2 - \le(\phi^{-2}-1\ri)\le(\trth-2\ri) + 4 \le(\phi^{-2}-1\ri)x^i\le(N(x^i)-x^i\ri) \\
      & -2 \le(N(x^i)-x^i\ri)^2 - \le(\trth-2\ri)\le(N(x^i)-x^i\ri)\le(N(x^i)+x^i\ri),
    \end{aligned}
  \end{align}
  where we note that the first term on the last line has a negative sign and can eventually be put on the left-hand side of~\eqref{eq:reproofBochner0}. Plugging~\eqref{eq:reproofBochner1}, \eqref{eq:reproofBochner2}, \eqref{eq:reproofBochner2Gauss}, \eqref{eq:reproofBochner3}, \eqref{eq:reproofBochner4} and~\eqref{eq:reproofBochner5} in \eqref{eq:reproofBochner0}, we get the desired~\eqref{eq:refinedBochner}.
\end{proof}

\section{Energy and Bochner estimates}\label{sec:enerBoch}
This section is dedicated to showing that the refined Bochner identity~\eqref{eq:refinedBochner} and the assumptions of Theorem~\ref{thm:main} yield the following estimates.
\begin{proposition}\label{prop:refinedBochnerest}
  Assume that the hypothesis of Theorem~\ref{thm:main} hold and let $\Phi$ be a conformal isomorphism between $\pr\Si$ and $\SSS^2$ with conformal factor $\phi$ given by Proposition~\ref{lem:unifthm2} and $x^i$ be the associated functions on $\Si$ of Definition~\ref{def:coordsxi}. We have
  \begin{align}\label{est:refinedBochner}
    \sum_{i=1}^3\norm{\nab^2x^i}^2_{L^2(\Si)} + \sum_{i=1}^3\norm{N(x^i)-x^i}_{H^{1/2}(\pr\Si)}^2 + \norm{\BBB}^2_{L^2(\Si)}  & \les_\AAA \varep^2,
  \end{align}
  and
  \begin{align}
    \label{est:consrefinedBochner}
    \sum_{i=1}^3\norm{\nab x^i}_{L^2(\Si)} + \sum_{i=1}^3\norm{\nab x^i}_{L^6(\Si)} & \les_\AAA 1,
  \end{align}
  provided that $\varep$ is sufficiently small.
\end{proposition}

The proof of Proposition~\ref{prop:refinedBochnerest} is postponed to the end of this section. It relies on the following lemmas.
\begin{lemma}
  We have
  \begin{align}\label{est:maxppl}
    |x^i| \leq 1,
  \end{align}
  for all $i=1,2,3$.
\end{lemma}
\begin{proof}
  This is a direct application of the maximum principle to Definition~\ref{def:coordsxi}.
\end{proof}
\begin{lemma}\label{lem:energyestnabxi}
  We have
  \begin{align}\label{est:energyestnabxibis}
    \sum_{i=1}^3\norm{\nab x^i}^2_{L^2(\Si)} & \les_\AAA 1 + \sum_{i=1}^3\norm{N(x^i)-x^i}_{L^2(\pr\Si)} \les_\AAA 1 + \sqrt{\EE},
  \end{align}
  and
  \begin{align}\label{est:energyestnabxi}
    \sum_{i=1}^3\norm{\nab x^i}^2_{L^6(\Si)} & \les_\AAA 1 + \sum_{i=1}^3\norm{N(x^i)-x^i}_{L^2(\pr\Si)} + \sum_{i=1}^3\norm{\nab^2x^i}^2_{L^2(\Si)} \les_\AAA 1 + \EE.
  \end{align}
\end{lemma}
\begin{proof}
  The first estimate follows from the energy identity~\eqref{eq:energyintro} and~\eqref{est:maxppl}, recalling the definition of $\EE$ from Proposition~\ref{prop:refinedBochner}. The second follows from the first and the Sobolev estimate~\eqref{est:sobeucl1bis}.
\end{proof}

\begin{lemma}\label{lem:H3/2Nxixi}
  We have
  \begin{align}
    \label{est:H3/2Nxixi}
    \sum_{i=1}^3\norm{\le(N(x^i)-x^i\ri)}^2_{H^{1/2}(\pr\Si)} & \les_\AAA \varep^2 + \EE,
  \end{align}
  and
  \begin{align}\label{est:H3/2Nxixibis}
    \sum_{i=1}^3\norm{x^i\le(N(x^i)-x^i\ri)}^2_{H^{1/2}(\pr\Si)} & \les_\AAA \varep^2 + \EE + \EE^2,
  \end{align}
  provided that $\varep$ is sufficiently small.
\end{lemma}
\begin{proof}
  Let $F$ be any smooth $\Si$-tangent vectorfield on $\prSi$ and denote by $\Fslash$ its projection on $T\prSi$. We have
  \begin{align}\label{eq:H3/2Nxi1}
    \begin{aligned}
      \int_{\pr\Si}F\cdot\Nd\le(N(x^i)-x^i\ri) & = \int_{\pr\Si}F^a\le(\Nd_a(N(x^i)) - \th_{ab}\Nd^bx^i +(\th_{ab}-\gd_{ab})\Nd^bx^i\ri) \\
      & = \int_{\pr\Si}\Fslash\cdot\nab\nab_Nx^i + \int_{\pr\Si}F^a(\th_{ab}-\gd_{ab})\Nd^b x^i.
    \end{aligned}
  \end{align}
  Let $\widetilde{F}$ be a smooth extension of $\Fslash$ in $\Si$. Applying Stokes theorem and using the commutation formula~\eqref{eq:DeltaRicbis}, we have
  \begin{align}\label{eq:H3/2Nxi2}
    \begin{aligned}
      \int_{\pr\Si}\Fslash\cdot\nab\nab_Nx^i = \int_{\pr\Si}\widetilde{F}\cdot\nab_N\nab x^i =\int_{\Si}\Div\le(\widetilde{F}\cdot\nab^2x^i\ri) = \int_{\Si} \nab \widetilde{F} \cdot \nab^2x^i + \int_{\Si}\widetilde{F}\cdot\underbrace{\Delta\nab x^i}_{=\RRRic\cdot\nab x^i}. 
    \end{aligned}
  \end{align}
  Combining~\eqref{eq:H3/2Nxi1} and~\eqref{eq:H3/2Nxi2}, using the bounds~\eqref{est:L2RicTh} on $\RRRic$ and $\th$ and the Sobolev embedding~\eqref{est:L2H12pf} on $\prSi$ and~\eqref{est:sobeucl1bis} on $\Si$, we get
  \begin{align*}
    \begin{aligned}
      \le|\int_{\pr\Si}F\cdot\Nd\le(N(x^i)-x^i\ri)\ri| & \les_\AAA \le(\varep\norm{\nab x^i}_{L^6(\Si)} + \norm{\nab^2x^i}_{L^2(\Si)}\ri)\Vert\widetilde{F}\Vert_{H^1(\Si)} + \varep \norm{\Nd x^i}_{L^4(\prSi)} \norm{F}_{H^{1/2}(\prSi)} \\
      & \les_\AAA \le(\varep\norm{\nab x^i}_{L^6(\Si)} + \norm{\nab^2x^i}_{L^2(\Si)}\ri)\Vert F\Vert_{H^{1/2}(\prSi)},
    \end{aligned}
  \end{align*}
  where we took the infimum on all extensions $\widetilde{F}$ of $\Fslash$, used definition~\eqref{eq:defH12} and Corollary~\ref{cor:decompperpH12}. Hence,
  \begin{align*}
    \norm{\Nd(N(x^i)-x^i)}_{H^{-1/2}(\pr\Si)} & \les_\AAA \norm{\nab^2 x^i}_{L^2(\Si)} + \varep \norm{\nab x^i}_{L^6(\Si)}.
  \end{align*}
  Applying~\eqref{est:H12L2H-12bis}, using the definition of $\EE$ from Proposition~\ref{prop:refinedBochner} and~\eqref{est:energyestnabxi}, we get
  \begin{align*}
    \norm{N(x^i)-x^i}^2_{H^{1/2}(\pr\Si)} & \les_\AAA \norm{N(x^i)-x^i}^2_{L^2(\pr\Si)} + \norm{\Nd(N(x^i)-x^i)}^2_{H^{-1/2}(\pr\Si)} \les_\AAA \EE +\varep^2,
  \end{align*}
  and the desired~\eqref{est:H3/2Nxixi} follows. Taking $f$ be a smooth extension of $N(x^i)-x^i$ on $\Si$, using~\eqref{est:sobeucl1bis},~\eqref{est:maxppl} and~\eqref{est:energyestnabxi}, we have
  \begin{align}\label{est:H12xiNxi}
    \begin{aligned}
      \norm{x^i(N(x^i)-x^i)}^2_{H^{1/2}(\prSi)} & \leq \norm{x^if}^2_{H^1(\Si)} \les_\AAA \norm{x^i}^2_{L^\infty(\Si)}\norm{\nab f}^2_{H^1(\Si)} + \norm{\nab x^i}^2_{L^6(\Si)}\norm{f}^2_{L^6(\Si)} \\
      & \les_\AAA \le(1+\EE\ri)\norm{N(x^i)-x^i}^2_{H^{1/2}(\prSi)},
    \end{aligned}
  \end{align}
  where, to obtain the last line, we took the infimum over all $f$. This proves the desired~\eqref{est:H3/2Nxixibis}.
\end{proof}

\begin{lemma}\label{lem:estL2BBB}
  We have
  \begin{align}\label{est:L2BBB}
    \norm{\BBB}_{L^2(\Si)}^2 & \les_\AAA \varep^2 + \EE + \EE^2,
  \end{align}
  provided that $\varep$ is sufficiently small.
\end{lemma}
\begin{proof}
  Let us define the vectorfield
  \begin{align*}
    Z & := \sum_{i=1}^3x^i\nab x^i.
  \end{align*}
  From~\eqref{eq:sumxi} and the definition of the harmonic radius vectorfield $X$ from Definition~\ref{def:X}, we have
  \begin{align*}
    Z-X = \le(\sum_{i=1}^3x^i(N(x^i)-x^i)\ri) N
  \end{align*}
  on $\prSi$. Using Corollary~\ref{cor:decompperpH12} and~\eqref{est:H3/2Nxixibis}, we have
  \begin{align}\label{est:H12ZXprSi}
    \norm{Z-X}^2_{H^{1/2}(\prSi)} & \les_\AAA \sum_{i=1}^3\norm{x^i(N(x^i)-x^i)}^2_{H^{1/2}(\prSi)} \les_\AAA \varep^2+\EE+\EE^2.
  \end{align}
  Moreover, for all smooth tensor $F$ on $\prSi$, noting $\widetilde{F}$ a smooth extension of $F$, using Stokes formula, Sobolev estimate~\eqref{est:sobeucl1bis} and taking the infimum over all $\widetilde{F}$ and using~\eqref{eq:defH12}, we have
  \begin{align*}
    \begin{aligned}
      \le|\int_\prSi F \cdot\nab_N(Z-X)\ri| & = \le|\int_\Si\Div\le(\widetilde{F}\cdot\nab(Z-X)\ri) \ri| \\
      & \les_\AAA \norm{F}_{H^{1/2}(\prSi)} \le(\norm{Z-X}_{H^1(\Si)} + \norm{\Delta(Z-X)}_{L^{6/5}(\Si)}\ri).
    \end{aligned}
  \end{align*}
  Hence
  \begin{align}\label{est:H-12ZX}
    \begin{aligned}
      \norm{\nab_N(Z-X)}_{H^{-1/2}(\prSi)} & \les_\AAA \norm{Z-X}_{H^1(\Si)} + \norm{\Delta(Z-X)}_{L^{6/5}(\Si)}.
    \end{aligned}
  \end{align}
  Using that $X$ is harmonic,~\eqref{eq:DeltaRicbis}, and~\eqref{est:maxppl}, we have
  \begin{align}\label{est:LapXZSi}
    \begin{aligned}
      \norm{\Delta(Z-X)}^2_{L^{6/5}(\Si)} & \les \sum_{i=1}^3\le(\norm{\nab x^i \nab^2x^i}^2_{L^{6/5}(\Si)} +\norm{ x^i \RRRic \cdot\nab x^i}^2_{L^{6/5}(\Si)}\ri) \\
      & \les_\AAA \sum_{i=1}^3\norm{\nab x^i}^2_{L^6(\Si)}\le(\norm{\nab^2x^i}^2_{L^2(\Si)} + \norm{\RRRic}_{L^2(\Si)}^2\ri) \\
      & \les_\AAA \varep^2+\EE+\EE^2, 
    \end{aligned}
  \end{align}
  where the last inequality was obtained using~\eqref{est:energyestnabxi}. We have
  \begin{align*}
    \begin{aligned}
      \norm{\nab(Z-X)}^2_{L^2(\Si)} & = -\int_\Si (Z-X)\cdot\Delta(Z-X) + \int_\prSi (Z-X)\cdot\nab_N(Z-X).
    \end{aligned}
  \end{align*}
  Hence, using the Sobolev estimate~\eqref{est:SobL6L2H12} and~\eqref{est:H12ZXprSi}, \eqref{est:H-12ZX} and \eqref{est:LapXZSi} and an absorption argument, we obtain
  \begin{align*}
    \norm{Z-X}_{L^6(\Si)}^2 + \norm{\nab(Z-X)}^2_{L^2(\Si)} & \les_\AAA \norm{\Delta(Z-X)}_{L^{6/5}(\Si)} \norm{Z-X}_{L^6(\Si)} \\
                                                            & \quad + \norm{Z-X}_{H^{1/2}(\prSi)}\norm{\nab_N\le(Z-X\ri)}_{H^{-1/2}(\prSi)} + \norm{Z-X}_{H^{1/2}(\prSi)}^2 \\
                                                            & \les_\AAA \varep^2+\EE+\EE^2.
  \end{align*}
  Using \eqref{est:maxppl} and estimate~\eqref{est:H2X} on $\nab X -g$, we conclude
  \begin{align*}
    \norm{\BBB}^2_{L^2(\Si)} & \les \sum_{i=1}^3\norm{x^i\nab^2x^i}_{L^2(\Si)}^2 + \norm{\nab Z-g}^2_{L^2(\Si)} \\
                             & \les \EE + \norm{\nab(Z-X)}_{L^2(\Si)}^2 + \norm{\nab X -g}_{L^2(\Si)}^2 \\
                             & \les_\AAA \varep^2+\EE+\EE^2,
  \end{align*}
  which proves the desired~\eqref{est:L2BBB}.
\end{proof}

\begin{lemma}\label{lem:Bochnerrefined1}
  We have
  \begin{align}
    \EE & \les \varep^2+\varep\EE.\label{est:Bochnerrefined1}
  \end{align}
\end{lemma}
\begin{proof}
  Using the assumptions~\eqref{est:L2RicTh} and~\eqref{est:L2BBB}, we have
  \begin{align}\label{est:RHSnab2xiB}
    \begin{aligned}
      -3\int_{\Si}\le(\RRRic-\frac{1}{3} (\tr_g\RRRic)\, g\ri)\cdot\BBB \les \norm{\RRRic}_{L^2(\Si)}\norm{\BBB}_{L^2(\Si)} \les\varep \sqrt{\varep^2+\EE+\EE^2}.
    \end{aligned}
  \end{align}  
  Using estimate~\eqref{est:H12GNN} for the Einstein tensor on $\pr\Si$ with~\eqref{est:L2RicTh}, and estimate~\eqref{est:H3/2Nxixibis}, we have
  \begin{align}\label{est:proofBochner1}
    \begin{aligned}
      \le|\int_{\pr\Si}  \mathrm{G}_{NN}x^i (N(x^i)-x^i)\ri| & \les_\AAA \norm{\mathrm{G}_{NN}}_{H^{-1/2}(\pr\Si)}\norm{x^i (N(x^i)-x^i)}_{H^{1/2}(\pr\Si)} \\
      & \les_\AAA \varep \sqrt{\varep^2+\EE+\EE^2}.
    \end{aligned}
  \end{align}
  Using \eqref{est:L2RicTh},~\eqref{est:maxppl}, the functional estimates~\eqref{est:L2H12pf}, the estimates~\eqref{est:assconf} for $\phi$ and the expression of $\EE$, we have
  \begin{align}\label{est:RHSnab2xiG}
    \begin{aligned}
      \le|\int_{\Si}x^i \mathrm{G}\cdot \nab^2x^i\ri| & \les \varep \norm{\nab^2x^i}_{L^2(\Si)} \les \varep \sqrt{\EE}, \\
      \int_{\pr\Si}(\trth-2)^2 + \int_{\pr\Si}|\thh|^2 & \les \norm{\th-\gd}^2_{L^2(\pr\Si)} \les_\AAA \varep^2,\\
      \le|\int_{\pr\Si}(\phi^{-2}-1)(\trth-2)\ri| & \les \norm{\th-\gd}_{L^2(\pr\Si)}\norm{\phi^{-2}-1}_{L^2(\pr\Si)} \les_\AAA \varep^2\\
      \le|\int_{\pr\Si}2x^i(N(x^i)-x^i)(1-\phi^{-2})\ri| & \les \norm{N(x^i)-x^i}_{L^2(\pr\Si)}\norm{\phi^{-2}-1}_{L^2(\pr\Si)} \les_\AAA \varep \sqrt{\EE}.
    \end{aligned}
  \end{align}
  Using additionally the trace estimate~\eqref{eq:defH12} and~\eqref{est:energyestnabxibis}, we have 
  \begin{align}\label{est:RHSnab2xiGbis}
    \begin{aligned}
       & \quad \le|\int_{\pr\Si}\le(\trth-2\ri) \le(N(x^i)-x^i\ri) \le(N(x^i)+x^i\ri)\ri| \\
      & \leq \norm{\trth-2}_{L^4(\pr\Si)} \norm{N(x^i)-x^i}_{L^2(\pr\Si)} \le(\norm{\nab x^i}_{L^4(\pr\Si)} + \norm{x^i}_{L^4(\prSi)}\ri) \\
      & \les_\AAA \varep \sqrt{\EE} \le(1+ \sqrt{1+\sqrt{\EE}}\ri).
    \end{aligned}
  \end{align}
  Examining the definition of $\EE$ from Proposition~\ref{prop:refinedBochner} and combining~\eqref{est:RHSnab2xiB}, \eqref{est:proofBochner1}, \eqref{est:RHSnab2xiG} and \eqref{est:RHSnab2xiGbis}, we get
  \begin{align*}
    \EE & \les_\AAA \varep \sqrt{\varep^2+\EE+\EE^2} + \varep\sqrt{\EE} + \varep^2 +  \varep \sqrt{\EE} \le(1+ \sqrt{1+\sqrt{\EE}}\ri) \les_\AAA \varep^2+\varep\EE,
  \end{align*}
  which proves the desired~\eqref{est:Bochnerrefined1}.
\end{proof}

\begin{proof}[Proof of Proposition~\ref{prop:refinedBochnerest}]
  By an absorption argument, we deduce from~\eqref{est:Bochnerrefined1} that $\EE\les_\AAA \varep^2$ which, using Proposition~\ref{prop:refinedBochner} and Lemmas~\ref{lem:energyestnabxi}, \ref{lem:H3/2Nxixi} and~\ref{lem:estL2BBB}, concludes the proof of the proposition.
\end{proof}

\section{Higher order estimates}\label{sec:nab3xiSi}
In this section, we prove estimates for $\nab^3x^i$. Estimates for higher derivatives will follow similarly and are left to the reader (see Remark~\ref{rem:higherorderlemmaell}).
\begin{proposition}\label{prop:nabxi3Si}
  Assume that the hypothesis of Theorem~\ref{thm:main} hold and let $\Phi$ be a conformal isomorphism between $\pr\Si$ and $\SSS^2$ with conformal factor $\phi$ given by Proposition~\ref{lem:unifthm2} and $x^i$ be the associated functions on $\Si$ of Definition~\ref{def:coordsxi}. We have
  \begin{align}\label{est:nab3xiSi}
    \norm{\nab^3x^i}_{L^2(\Si)} + \norm{\nab^2x^i}_{L^6(\Si)} & \les_\AAA \varep,
  \end{align}
  and
  \begin{align}\label{est:Linftynabxi}
    \norm{\nab x^i}_{L^\infty(\Si)} & \les_\AAA 1,
  \end{align}
  for all $1\leq i\leq 3$, and
  \begin{align}\label{est:LinftyBBB}
    \norm{\BBB}_{H^2(\Si)} + \norm{\BBB}_{L^\infty(\Si)} & \les_\AAA \varep.
  \end{align}
\end{proposition}

The proof of Proposition~\ref{prop:nabxi3Si} is obtained by standard higher order elliptic estimates for (commutations of) Laplace equation. It relies on the following two lemmas which we kept as general as possible to underline how this can be generalised to obtain higher order estimates. See also Remark~\ref{rem:higherorderlemmaell}. The proof of Proposition~\ref{prop:nabxi3Si} itself is postponed to the end of this section. 
\begin{lemma}\label{lem:ellipticnab3xi}
  Under the assumptions of Theorem~\ref{thm:main}, provided that $\varep$ is sufficiently small, we have
  \begin{align}\label{est:enercoordsxi2new}
    \begin{aligned}
      \norm{\nab^3x^i}_{L^2(\Si)}^2 & \les_{\AAA} \norm{\Delta \nab x^i}^2_{L^2(\Si)} + \varep\norm{\nab^2x^i}^2_{L^2(\Si)} + \varep^2\norm{\nab x^i}^2_{L^2(\Si)} + \big\Vert\slashed{H}^i\big\Vert^2_{H^{1/2}(\prSi)},
    \end{aligned}
  \end{align}
  for all $1\leq i\leq 3$ and where $\slashed{H}^i$ is the projection on $\le(T^\ast(\prSi)\ri)^{\otimes2}$ of $\nab^2x^i$.
\end{lemma}
\begin{lemma}\label{lem:controlslashedH}
  Under the assumptions of Theorem~\ref{thm:main}, provided that $\varep$ is sufficiently small, we have
  \begin{align}
    \label{est:H12nabnabxi2}
    \big\Vert\slashed{H}^i\big\Vert_{H^{1/2}(\prSi)} & \les_\AAA \norm{\phi-1}_{H^{3/2}(\prSi)} + \norm{\th-\gd}_{H^{1/2}(\prSi)} + \norm{N(x^i)-x^i}_{H^{1/2}(\prSi)} + \varep\norm{\nab x^i}_{H^2(\Si)},
  \end{align}
  for all $1\leq i\leq 3$.
\end{lemma}

\begin{proof}[Proof of Lemma~\ref{lem:ellipticnab3xi}]
  Taking $U=\nab x^i$ in the elliptic estimate~\eqref{est:highelliptic}, we get
  \begin{align}\label{est:nabxi31}
    \begin{aligned}
      \norm{\nab^3x^i}^2_{L^2(\Si)} & \les_\AAA \norm{\Delta\nab x^i}_{L^2(\Si)}^2 + \varep\norm{\nab^2x^i}^2_{L^{2}(\Si)} + \varep^2 \norm{\nab x^i}^2_{L^2(\Si)} + \BB,
    \end{aligned}
  \end{align}
  with
  \begin{align*}
    \BB & := \int_{\pr\Si}\le(\nab_a\nab_bx^i\nab^a\nab_N\nab^bx^i-\nab^a\nab_a\nab_b x^i\nab_N\nab^b x^i\ri) \\
        & \quad + \int_\prSi\le(\nab_a\nab_Nx^i\nab^a\nab_N\nab_Nx^i-\nab^a\nab_a\nab_N x^i\nab_N\nab_N x^i\ri) \\
        & = \int_{\pr\Si}\le(\nab_a\nab_bx^i\nab^a\nab_N\nab^bx^i-\nab^a\nab_a\nab_b x^i\nab_N\nab^b x^i\ri) \\
        & \quad + \int_\prSi\le(-\nab_a\nab_Nx^i\nab^a\nab^b\nab_bx^i+\nab^a\nab_a\nab_N x^i\nab^b\nab_b x^i\ri),
  \end{align*}
  where we used that $\Delta x^i=0$. In the sequel we fix $1\leq i\leq 3$ and we define $\slashed{L}$ to be the $\prSi$-tangent tensor such that $\slashed{L}_{ab}=\nab_a\nab_N\nab_bx^i$. To avoid confusion between indices, we write $\slashed{H} := \slashed{H}^i$. With these notations, the four terms composing $\BB$ rewrite as
  \begin{align*}
    \int_\prSi\nab_a\nab_bx^i\nab^a\nab_N\nab^bx^i & = \int_\prSi \slashed{H}\cdot\slashed{L},
  \end{align*}
  and
  \begin{align*}
    \begin{aligned}
      -\int_{\prSi}\nab^a\nab_a\nab_bx^i\nab_N\nab^bx^i & = -\int_{\prSi}\le(\Nd^a\slashed{H}_{ab} +  \trth \nab_N\nab_b x^i + \th^{a}_b\nab_a\nab_N x^i\ri)\nab_N\nab^bx^i \\
      & = \int_{\prSi}\slashed{H}_{ab}\big(\nab^a\nab_N\nab^bx^i+\th^{ac}\nab_c\nab^bx^i - \th^{ab}\nab_N \nab_Nx^i\big) \\
      & \quad\quad\quad - \trth |\Nd\nab_Nx^i|^2 -\th^{ab}\Nd_a\nab_Nx^i\Nd_b\nab_Nx^i\\
      & = \int_{\prSi}\slashed{H}_{ab}\slashed{L}^{ab}+\th^{ac}\slashed{H}_{ab}\slashed{H}^b_{c} + \th^{ab}\slashed{H}_{ab}\tr\slashed{H} \\
      & \quad\quad\quad - \trth |\Nd\nab_Nx^i|^2 -\th^{ab}\Nd_a\nab_Nx^i\Nd_b\nab_Nx^i,
    \end{aligned}
  \end{align*}
  and
  \begin{align*}
    \begin{aligned}
      \int_\prSi\nab^a\nab_a\nab_N x^i\nab^b\nab_b x^i & = \int_\prSi \tr\slashed{H}\tr\slashed{L},
    \end{aligned}
  \end{align*}
  and
  \begin{align*}
    \begin{aligned}
      -\int_{\prSi}\nab^a\nab^b\nab_bx^i\nab_a\nab_Nx^i & = \int_{\prSi}\le(-\Nd^a\tr\slashed{H} - 2\th^{ab}\nab_N\nab_bx^i\ri)\nab_a\nab_Nx^i \\
      & = \int_\prSi \tr\slashed{H}\le(\nab^a\nab_a\nab_Nx^i - \trth\nab_N\nab_Nx^i + \th^{ab}\nab_a\nab_bx^i\ri) \\
      & \quad\quad\quad - 2\th^{ab}\nab_N\nab_bx^i\nab_N\nab_ax^i \\
      & = \int_\prSi \tr\slashed{H}\tr\slashed{L} + \trth(\tr\slashed{H})^2 + \th^{ab}\slashed{H}_{ab}\tr\slashed{H} \\
      & \quad\quad\quad - 2\th^{ab}\nab_N\nab_bx^i\nab_N\nab_ax^i.
    \end{aligned}
  \end{align*}
  Thus, using the functional estimates~\eqref{est:L2H12pf} and the bound~\eqref{est:L2RicTh} on $\th-\gd$,
  \begin{align}\label{est:BBgoodfinal}
    \begin{aligned}
      \BB + 8 \int_\prSi |\Nd\nab_Nx^i|^2 & \les_\AAA \norm{\slashed{H}}_{H^{1/2}(\prSi)}\norm{\slashed{L}}_{H^{-1/2}(\prSi)} + \norm{\slashed{H}}^2_{H^{1/2}(\prSi)} + \varep \norm{\Nd\nab_Nx^i}_{H^{1/2}(\prSi)}^2.
    \end{aligned}
  \end{align}
  Applying the trace estimate~\eqref{est:traceH-12N} and Corollary~\ref{cor:decompperpH12}, we get
  \begin{align}\label{est:H-12slashedL}
    \Vert\slashed{L}\Vert_{H^{-1/2}(\prSi)} & \les_\AAA \norm{\Nd\nab_N\nab x^i}_{H^{-1/2}(\prSi)} \les_\AAA \norm{\nab^3x^i}_{L^2(\Si)} + \varep\norm{\nab^2x^i}_{L^2(\Si)},
  \end{align}
  and from~\eqref{eq:defH12} and Corollary~\ref{cor:decompperpH12}, we have
  \begin{align}\label{est:smallbitNdnabxiH12}
    \norm{\Nd\nab_Nx^i}_{H^{1/2}(\prSi)} & \les_\AAA \norm{\nab^3x^i}_{L^2(\Si)} + \norm{\nab^2x^i}_{L^2(\Si)}.
  \end{align}
  Plugging~\eqref{est:H-12slashedL} and~\eqref{est:smallbitNdnabxiH12} in the estimate~\eqref{est:BBgoodfinal} for $\BB$, and plugging this in~\eqref{est:nabxi31}, we get
  \begin{align*}
    \begin{aligned}
      \norm{\nab^3x^i}^2_{L^2(\Si)} & \les_\AAA \norm{\Delta\nab x^i}_{L^2(\Si)}^2 + \varep\norm{\nab^2x^i}^2_{L^{2}(\Si)} + \varep^2 \norm{\nab x^i}^2_{L^2(\Si)} + \norm{\slashed{H}}_{H^{1/2}(\prSi)}^2 \\
      & \quad + \norm{\slashed{H}}_{H^{1/2}(\prSi)}\le(\norm{\nab^3x^i}_{L^2(\Si)} + \varep\norm{\nab^2x^i}_{L^2(\Si)}\ri) \\
      & \quad + \varep \le(\norm{\nab^3x^i}_{L^2(\Si)} + \norm{\nab^2x^i}_{L^2(\Si)}\ri)^2,
    \end{aligned}
  \end{align*}
  which, together with an absorption argument, finishes the proof of the desired~\eqref{est:enercoordsxi2new}.  
\end{proof}
\begin{proof}[Proof of Lemma~\ref{lem:controlslashedH}]
  Using the conformal formula~\eqref{eq:confHessianxi}, we have
  \begin{align}\label{eq:nabnabxiNdNd}
    \begin{aligned}
    \slashed{H}_{ab} = \nab_a\nab_bx^i & = \Nd_a\Nd_bx^i + \th_{ab}N(x^i) \\
    & = x^i(1-\phi^{-2})\gd_{ab} - 2(\Nd(\log\phi)\otimesh\Nd x^i)_{ab} + x^i(\th_{ab}-\gd_{ab}) \\
    & \quad + (N(x^i)-x^i)\gd_{ab} + (N(x^i)-x^i)(\th_{ab}-\gd_{ab}).
    \end{aligned}
  \end{align}
  From~\eqref{eq:nabnabxiNdNd}, using the functional estimates~\eqref{est:L2H12pf}, \eqref{est:H12L2H-12bis} and~\eqref{est:productLinftyH12}, and the estimates~\eqref{est:L2RicTh}, \eqref{est:gdH12}, \eqref{est:assconf}, \eqref{est:maxppl} and~\eqref{est:refinedBochner}, \eqref{est:consrefinedBochner}, one has
  \begin{align*}
    \begin{aligned}
      \Vert\slashed{H}\Vert_{H^{1/2}(\prSi)} & \les_\AAA \norm{x^i\gd}_{L^\infty(\prSi)}\norm{\phi^{-2}}_{L^\infty(\prSi)}\norm{\phi+1}_{L^\infty(\prSi)}\norm{\phi-1}_{L^\infty(\prSi)} \\
      & \quad + \norm{x^i\gd}_{L^\infty(\prSi)}\norm{\phi^{-2}}_{L^\infty(\prSi)}\norm{\phi+1}_{L^\infty(\prSi)}\norm{\Nd(\phi-1)}_{L^2(\prSi)} \\
      & \quad + \norm{\Nd x^i}_{L^2(\prSi)} \norm{\phi^{-2}}_{L^\infty(\prSi)}\norm{\phi+1}_{L^\infty(\prSi)}\norm{\phi-1}_{L^\infty(\prSi)} \\
                                                  & \quad + \norm{\Nd\log\phi}_{H^{1/2}(\prSi)}\norm{\Nd x^i}_{H^{3/2}(\prSi)} + \norm{x^i}_{H^{3/2}(\prSi)}\norm{\th-\gd}_{H^{1/2}(\prSi)} \\
                                                  & \quad + \norm{N(x^i)-x^i}_{H^{1/2}(\prSi)}\norm{\gd}_{H^{3/2}(\prSi)} \\
                                                  & \quad + \norm{\th-\gd}_{H^{1/2}(\prSi)}\norm{N(x^i)-x^i}_{H^{3/2}(\prSi)} \\
                                                  & \les_\AAA \norm{\phi-1}_{H^{3/2}(\prSi)} + \norm{\th-\gd}_{H^{1/2}(\prSi)} + \norm{N(x^i)-x^i}_{H^{1/2}(\prSi)} \\
                                                  & \quad + \varep\norm{\Nd x^i}_{H^{3/2}(\prSi)} + \varep\norm{N(x^i)-x^i}_{H^{3/2}(\prSi)}.
    \end{aligned}
  \end{align*}
  Note that, to obtain the above, one used that~\eqref{est:gdH12} yields 
  \begin{align*}
    \norm{\gd}_{H^{3/2}(\prSi)} & = \norm{\gd}_{H^{1/2}(\prSi)} \les_\AAA 1. 
  \end{align*}
  Using~\eqref{est:L2RicTh}, Corollary~\ref{cor:decompperpH12} and~\eqref{est:productLinftyH12}, one can moreover easily show that
  \begin{align*}
    \norm{\Nd x^i}_{H^{3/2}(\prSi)} + \norm{N(x^i)-x^i}_{H^{3/2}(\prSi)} & \les_\AAA \norm{\nab x^i}_{H^2(\Si)}, 
  \end{align*}
  and this concludes the proof of the desired~\eqref{est:H12nabnabxi2}.
\end{proof}

\begin{proof}[Proof of Proposition~\ref{prop:nabxi3Si}]
  Using~\eqref{eq:DeltaRicbis} and the Sobolev estimate~\eqref{est:sobeucl1bis}, we have
  \begin{align*}
    \norm{\Delta \nab x^i}_{L^2(\Si)} \les \norm{\RRRic}_{L^2(\Si)} \norm{\nab x^i}_{L^\infty(\Si)} \les_\AAA \varep \norm{\nab x^i}_{H^2(\Si)}.
  \end{align*}
  From~\eqref{est:H12nabnabxi2}, together with the estimates~\eqref{est:L2RicTh}, \eqref{est:assconf}, and~\eqref{est:refinedBochner}, we get
  \begin{align*}
    \norm{\slashed{H}}_{H^{1/2}(\prSi)} & \les_\AAA \varep + \varep\norm{\nab x^i}_{H^2(\Si)}. 
  \end{align*}
  Plugging the above two estimates in~\eqref{est:enercoordsxi2new}, using Sobolev estimates~\eqref{est:sobeucl1bis}, \eqref{est:sobeucl2bis}, and using~\eqref{est:refinedBochner}, together with an absorption argument, we obtain the desired estimates~\eqref{est:nab3xiSi} and~\eqref{est:Linftynabxi}. For $\BBB$, we have
  \begin{align*}
    \nab_j \BBB_{kl} & = \sum_{i=1}^3\nab_k x^i \nab_j\nab_lx^i + \sum_{i=1}^3\nab_l x^i \nab_{j}\nab_kx^i.
  \end{align*}
  Hence, using~\eqref{est:nab3xiSi} and~\eqref{est:Linftynabxi}, we infer
  \begin{align*}
    \norm{\nab\BBB}_{L^2(\Si)} & \les \norm{\nab x^i}_{L^\infty(\Si)} \norm{\nab^2x^i}_{L^2(\Si)} \les_\AAA \varep,
  \end{align*}
  and similarly
  \begin{align*}
    \norm{\nab^2\BBB}_{L^2(\Si)} & \les \norm{\nab x^i}_{L^\infty(\Si)}\norm{\nab^3x^i}_{L^2(\Si)} + \norm{\nab^2x^i}^2_{L^4(\Si)} \les_\AAA \varep.
  \end{align*}
  The desired~\eqref{est:LinftyBBB} then follows from \eqref{est:refinedBochner} and the Sobolev estimate~\eqref{est:sobeucl2bis}.
\end{proof}

\begin{remark}\label{rem:higherorderlemmaell}
  To obtain higher order estimates for $\nab^{4}x^i$, \emph{etc.} one can apply the elliptic estimate~\eqref{est:highelliptic} to $U=\nab^2x^i$ (instead of $U=\nab x^i$ above). This involves higher order commutation of the Laplace equation, and, arguing as in the proof of Lemma~\ref{lem:ellipticnab3xi}, a higher order boundary term $\slashed{H}$. Following the proof of Lemma~\ref{lem:controlslashedH}, one can control that boundary term by higher order derivatives of $\th-\gd$ and $\phi-1$ (which is itself controlled by $\th-\gd$ and $\mathrm{R}$, see~\eqref{est:higherregconf}). Higher order commutations of Laplace equation involve higher order derivatives of the curvature tensor $\Riem$, and we claim that this yields the higher order estimates of Theorem~\ref{thm:main}.
\end{remark}

\section{Geometric conclusions}\label{sec:globdiffeo}
\begin{proposition}\label{prop:globdiffeo}
  Assume that there exists a conformal isomorphism $\Phi$ between $\pr\Si$ and $\SSS^2$ (see Definition~\ref{def:confiso}). Let $x^i$ be the associated functions on $\Si$ of Definition~\ref{def:coordsxi} and extend the definition of $\Phi$ by
  \begin{align*}
  \Phi~:~x\in\Si \mapsto \le(x^1(x),x^2(x),x^3(x)\ri) \in \mathbb{R}^3.
  \end{align*}
  Assume that
  \begin{align*}
    |\BBB| & = \le|\sum_{i=1}^3\nab x^i\otimes\nab x^i-g\ri| < 1,
  \end{align*}
  uniformly on $\Si$. Then, $\Phi$ is a global diffeomorphism from $\Si$ onto the unit 3-ball $\DDD^3\subset\mathbb{R}^3$.
\end{proposition}

The proof of Proposition~\ref{prop:globdiffeo} directly follows from the following three lemmas.
\begin{lemma}\label{lem:localinversefunction}
  Assume that $|\BBB|<1$ at $p\in\Si$. Then, $\Phi$ is a local diffeomorphism on a neighbourhood of $p$.
\end{lemma}
\begin{lemma}\label{lem:glob1}
  We have $\Phi\left(\Si\right) = \DDD^3$. Subsequently, $\Phi:\Si\to\DDD^3$ is a covering of $\DDD^3$.
\end{lemma}
\begin{lemma}\label{lem:glob2}
  All covering maps $\Psi:X\to Y$, with $X$ connected and $Y$ simply connected, are globally injective. 
\end{lemma}
\begin{proof}[Proof of Lemma~\ref{lem:localinversefunction}]
  By the spectral theorem for symmetric matrices, $|\BBB|^2 = (\la_1-1)^2 + (\la_2-1)^2 + (\la_3-1)^2$ where $\la_1, \la_2, \la_3$ are the eigenvalues of the map $\Theta:Z\in T_p\Si \mapsto \sum_{i=1}^3(\nab x^i)^\sharp\nab_Z x^i \in T_p\Si$. The condition $|\BBB|<1$ thus implies that $\la_1,\la_2,\la_3>0$ and that the map $\Theta$ is an isomorphism of $T_p\Si$. From the definition of $\Theta$, we have $\Theta(T_p\Si) \subset \mathrm{Span}\le((\nab x^1)^\sharp,(\nab x^2)^\sharp,(\nab x^3)^\sharp\ri)$. Thus $\le((\nab x^1)^\sharp,(\nab x^2)^\sharp,(\nab x^3)^\sharp\ri)$ is a basis of $T_p\Si$ and the conclusion of the lemma follows from the inverse function theorem.
\end{proof}
\begin{proof}[Proof of Lemma~\ref{lem:glob1}]
  Let first prove that $\Phi(\Si) \subset  \DDD^3$. We argue by contradiction and suppose that there exists $p\in\Si$ such that $|\Phi(p)|>1$.
  Let define the function $\chi$ on $\Si$ by
  \begin{align*}
    \chi := \frac{1}{|\Phi(p)|}\sum_{i=1}^n x^i(p) x^i.
  \end{align*}
  From the definitions of the $x^i$, the function $\chi$ is harmonic, $|\chi|\leq 1$ on $\pr\Si$ and $|\chi(p)| = |\Phi(p)|>1$, which contradicts the maximum principle.\\
  We then show that $\Phi(\Si) = \DDD^3$. Since $\Phi(\Si)$ is closed in $\DDD^3$ and $\DDD^3$ is connected, the result will follow provided that we can prove that $\Phi(\Si)$ is an open subset of $\DDD^3$. By definition of $\Phi$, we have on the one hand $\Phi(\pr\Si) = \pr\DDD^3$. Applying the maximum principle, we have on the other hand that $\Phi^{-1}\left(\pr\DDD^3\right) = \pr\Si$. Applying the inverse function theorem\footnote{To apply the inverse function theorem at the boundary, one extends the smooth function by Borel's lemma and then applies the classical inverse function theorem.} -- since $\Phi$ is a local diffeomorphism -- then ensures that $\Phi\left(\Si\right)$ is open in $\DDD^3$. By the inverse function theorem, any smooth surjective local diffeomorphism between compact manifolds defines a covering map.
\end{proof}
\begin{proof}[Proof of Lemma~\ref{lem:glob2}]
  This is standard algebraic topology. Let $p,q\in X$ such that $\Psi(p)=\Psi(q)=:x$. Since $X$ is connected, there exists a curve $\ga:[0,1]\to X$ such that $\ga(0)=p$, $\ga(1)=q$. Since $Y$ is simply connected, $\Psi\circ\ga$ can be contracted to the point $x$. Since $\Psi$ is a covering map, this implies that the curve $\ga$ is homotopic to a curve $\ga'$ which only takes values in $\Psi^{-1}(\{x\})$ and this homotopy has fixed ends $p$ and $q$. Since $\Psi$ is a covering map, $\ga'$ is constant, hence $p=q$.
\end{proof}

\section{Proof of Theorem~\ref{thm:main}}\label{sec:finalproof}
Under the assumptions of the theorem, we have constructed a conformal isomorphism $\Phi:p\in\prSi\mapsto(x^1(p),x^2(p),x^3(p))\in \SSS^2\subset\mathbb{R}^3$ (Proposition~\ref{lem:unifthm2}) which extends as a harmonic map $\Phi:p\in\Si \mapsto \le(x^1(p),x^2(p),x^3(p)\ri) \in \mathbb{R}^3$ (Definition~\ref{def:coordsxi}). From Proposition~\ref{lem:unifthm2}, Propositions~\ref{prop:refinedBochnerest} and~\ref{prop:nabxi3Si} we have obtained, in particular, that
\begin{align}\label{est:BBBconclusion}
  \norm{\BBB}_{L^\infty(\Si)} + \norm{\BBB}_{H^2(\Si)} & \les_\AAA \varep,
\end{align}
provided that $\varep$ is sufficiently small depending on $\AAA$ and where we recall that $\BBB = \sum_{i=1}^3\nab x^i\otimes\nab x^i -g$. Hence, by Proposition~\ref{prop:globdiffeo}, the map $\Phi$ is a global diffeomorphism from $\Si$ onto $\DDD^3$ and $\le(x^1,x^2,x^3\ri)$ are global coordinates on $\Si$. Let $\pr_1, \pr_2, \pr_3$ be the associated coordinate vectorfields. For all $1\leq i,j\leq 3$, one has
\begin{align}\label{eq:gijdeijBij}
  g_{ij}-\de_{ij} & = -\BBB_{ij}.
\end{align}
Hence, from~\eqref{est:BBBconclusion}, one has
\begin{align*}
  \le|g_{ii}-1\ri| \les |\BBB| g_{ii} \les_\AAA \varep g_{ii},
\end{align*}
thus,
\begin{align}\label{est:gii-1}
  \le|g_{ii} -1\ri| \les_\AAA \varep,
\end{align}
uniformly in $\Si$. Using~\eqref{est:gii-1} in~\eqref{eq:gijdeijBij}, we have
\begin{align*}
  \le|g_{ij}-\de_{ij}\ri| & \les_\AAA |\BBB| \sqrt{g_{ii}g_{jj}} \les_\AAA \varep,
\end{align*}
uniformly on $\Si$. Moreover, from~\eqref{eq:gijdeijBij}, we have schematically
\begin{align*}
  \nab (g_{ij}) & = (\nab \BBB)_{ij} + \BBB \cdot \nab(g_{kl}),  
\end{align*}
hence, using~\eqref{est:BBBconclusion} and~\eqref{est:gii-1} and an absorption argument, we have
\begin{align*}
  \max_{1\leq i,j\leq 3} \norm{\nab(g_{ij})}_{L^2(\Si)} \les_\AAA \norm{\nab \BBB}_{L^2(\Si)} \les_\AAA \varep, 
\end{align*}
and similarly
\begin{align*}
  \max_{1\leq i,j\leq 3} \norm{\nab^2(g_{ij})}_{L^2(\Si)} \les_\AAA \norm{\nab^2 \BBB}_{L^2(\Si)} \les_\AAA \varep,
\end{align*}
which proves the desired~\eqref{est:gijdeijthm}. We claim that the higher order estimates~\eqref{est:highergijdeijthm} follow similarly and this finishes the proof of Theorem~\ref{thm:main}.

\bibliographystyle{graf_GR_alpha}
\bibliography{graf_GR}
\end{document}